\documentclass[amssymb,amsfonts,12pt]{amsart}
\newtheorem{theorem}{Theorem}[section]
\newtheorem{lemma}[theorem]{Lemma}
\newtheorem{corollary}[theorem]{Corollary}
\newtheorem{proposition}[theorem]{Proposition}

\newtheorem*{ack*}{Acknowledgment}

\usepackage{graphicx,amscd,amsthm,amsfonts,amsopn,amssymb,verbatim}

\def\x{{\bf x}}

\def\F{{\mathcal F}}\def\R{{\mathbb R}}

\def\C{{\mathbb C}}
\def\B{{\mathcal B}}\def\nint{\mathop{\diagup\kern-13.0pt\int}}

\def\P{{\mathcal P}}\def\J{{\mathcal J}}
\def\M{{\mathfrak M}}\def\m{{\mathfrak m}}

\def\W{{\mathcal W}}

\def\dim{{\operatorname{dim}}}
\def\Avg{Avg}

\def\bas{\begin{align*}}
\def\eas{\end{align*}}
\def\bi{\begin{itemize}}
\def\ei{\end{itemize}}

\def\emph#1{{\it #1}}

\begin{document}
\author{Jean Bourgain}
\address{School of Mathematics, Institute for Advanced Study, Princeton NJ}
\email{bourgain@@math.ias.edu}
\author{Ciprian Demeter}
\address{Department of Mathematics, Indiana University,  Bloomington IN}
\email{demeterc@@indiana.edu}
\author{Larry Guth}
\address{Department of Mathematics, MIT, Cambridge MA}
\email{lguth@@math.mit.edu}

\thanks{The first author is partially supported by the NSF grant DMS-1301619. The second  author is partially supported  by the NSF Grant DMS-1161752. The third author is supported by a Simons Investigator award}
\title[Proof of Vinogradov's mean value theorem]{Proof of the main conjecture in Vinogradov's mean value theorem for degrees higher than three}

\begin{abstract}
We prove the main conjecture in Vinogradov's Mean Value Theorem for degrees higher than three. This will be a consequence of a sharp decoupling inequality for curves.
\end{abstract}
\maketitle

\section{Introduction}
For each integers $s\ge 1$ and  $n,N\ge 2$ denote by  $J_{s,n}(N)$ the number of integral solutions for the following system
$$X_1^i+\ldots+X_s^i=X_{s+1}^i+\ldots+X_{2s}^i,\;\;1\le i\le n,$$
with $1\le X_1,\ldots,X_{2s}\le N$.
The number $J_{s,n}(N)$ has the following analytic representation
$$J_{s,n}(N)=\int_{[0,1]^n}|\sum_{j=1}^Ne(x_1j+x_2j^2+\ldots+ x_n j^n)|^{2s}dx_1\ldots dx_n.$$
Here and throughout the rest of the paper we will write  $$e(z)=e^{2\pi i z},\;z\in\R.$$

Our main result  is the proof of the so called main conjecture in Vinogradov's Mean Value Theorem. Apart from the $N^\epsilon$ loss, this bound has been known to be sharp. The case $n=2$  follows easily from elementary estimates for the divisor function.\begin{theorem}
\label{ntmain1}
For each $s\ge 1$ and  $n,N\ge 2$ we have the upper bound
$$J_{s,n}(N)\lesssim_\epsilon N^{s+\epsilon}+N^{2s-\frac{n(n+1)}{2}+\epsilon}.$$
\end{theorem}
  In spite of its name, this result was in fact a conjecture until fairly recently, for $n\ge 3$. The case $n=3$ was solved by Wooley in \cite{Wo}, a major achievement  of his remarkable efficient congruencing method. Variants of this method also led to significant progress for larger values of $n$, see \cite{FoWo}, \cite{Wo1}. We refer the reader to the survey paper \cite{Wo2} for a description of related results. This  paper is also an excellent reference for the known consequences of Theorem \ref{ntmain1}. See also \cite{Wo3} for connections to the asymptotic formula in Waring's problem.

The $N^\epsilon$ loss in Theorem \ref{ntmain1} can be removed for $p>n(n+1)$, see \eqref{hghgjj3} below.
\bigskip

All the progress  on Vinogradov's Mean Value Theorem has so far relied on number theoretic methods.  We will take a different approach here, one that relies solely on harmonic analysis techniques. The consequences of our approach will be far more general than Theorem \ref{ntmain1}. In particular we will see that integers can be replaced with arbitrary well separated real numbers. The relevant machinery that we now call {\em decouplings}, showed its initial potential for applications in  papers such as \cite{TWol}, \cite{GarSe2}, \cite{PrSe} and \cite{Bo2}. Its full strength became apparent in the recent joint work \cite{BD3} of the first two authors, where sharp results were proved for hyper-surfaces. The decoupling theory has since proved to be a very successful tool for a wide variety of problems in number theory that involve exponential sums. See  \cite{Bo}, \cite{Bo6}, \cite{BD6},\cite{BD4}, \cite{BD5} and \cite{BW}. In particular, a line of attack on Vinogradov's Mean Value Theorem was initiated in \cite{BD4} and here we will rely on some of the tools that were developed there.
\bigskip

For a positive weight $v:\R^n\to[0,\infty)$ and for $f:\R^n\to\C$ we define the weighted integral
$$\|f\|_{L^p(v)}=(\int_{\R^n}|f(x)|^pv(x)dx)^{1/p}.$$
Also, for each ball $B=B(c_B,R)$ in $\R^n$ centered at $c_B$ and with radius $R$, we will introduce the weight
$$w_{B}(x)= \frac{1}{(1+\frac{|x-c_B|}{R})^{100n}}.$$
We denote by $\Gamma=\Gamma_n$ the curve
\begin{equation}
\label{ryvbt78y783956t787034r894}
\Gamma=\{\Phi(t)=(t,t^2,\ldots,t^n):0\le t\le 1\}.
\end{equation}
Given $g:[0,1]\to \C$ and an interval $J\subset [0,1]$,  we define the {\em extension} operator $E_J=E^{(n)}_J$ in $\R^n$ as follows
$$E_Jg(x)=\int_Jg(t)e(tx_1+t^2x_2+\ldots +t^nx_n)dt,$$
with
$$x=(x_1,\ldots,x_n).$$
Theorem \ref{ntmain1} will follow from the following general result. The connection is explained in Section \ref{decexp}.
\begin{theorem}
\label{ntmain}
Let $n\ge 2$ and $0<\delta\le 1$.
For each ball $B\subset \R^n$ with radius $\delta^{-n}$ and each $g:[0,1]\to\C$ we have
$$\|E_{[0,1]}g\|_{L^{n(n+1)}(w_B)}\lesssim_{\epsilon}\delta^{-\epsilon}(\sum_{J\subset [0,1]\atop{|J|=\delta}}\|E_Jg\|_{L^{n(n+1)}(w_B)}^2)^{1/2}.$$
The implicit constant is independent of $\delta,B,g$.
\end{theorem}
A sum such as
$\sum_{J\subset [0,1]\atop{|J|=\delta}}$
will always be understood when the reciprocal of $\delta$ is an integer, and $J$ ranges over the intervals $[j\delta,(j+1)\delta]$, $0\le j\le \delta^{-1}-1$.

\begin{ack*}
The authors thank T. Wooley for stimulating discussions over the last two years and they thank T. Tao for a careful reading of an earlier version of the manuscript and for pointing out a few typos. They also thank  G. Liu for some mumerics related to the theorem in the Appendix. The second author is grateful to his student Fangye Shi for a few suggestions that have increased the readability of the manuscript.
\end{ack*}
\bigskip

\section{Notation}
\bigskip

Throughout the paper we will write $A\lesssim_{\upsilon}B$ to denote the fact that $A\le CB$ for a certain implicit constant $C$ that depends on the parameter $\upsilon$. Typically, this parameter is either $\epsilon$ or $K$. The implicit constant will never depend on the scale $\delta$, on the balls we integrate over, or on the function $g$. It will however most of the times depend on the degree $n$ and on the Lebesgue index $p$. Since these can  be thought of as being fixed parameters, we will in general not write $\lesssim_{p,n}$.

We will denote by $B_R$ an arbitrary ball of radius $R$. We use the following two notations for averaged integrals
$$\nint_B F=\frac1{|B|}\int_BF,$$
$$\|F\|_{L^p_\sharp(w_B)}=(\frac1{|B|}\int|F|^pw_B)^{1/p}.$$
$|A|$ will refer to either the cardinality of $A$ if $A$ is finite, or to its Lebesgue measure if $A$ has positive measure.

For an interval $J\subset [0,1]$, we will write
$$\Gamma_J=\{(t,t^2,\ldots,t^n):\;t\in J\}.$$

\bigskip

\section{Overview of the method}

The proof of Theorem \ref{ntmain} builds on  significant progress recorded over the last ten years in an area of harmonic analysis called {\em restriction theory}. This area that emerged in the late 1960s was initially concerned with understanding the  $L^p$ norms of the Fourier transforms of measures supported on hyper-surfaces. It has since grown into a field with fascinating connections to incidence geometry and with far-reaching consequences in PDEs and number theory.
\bigskip

There are three major ingredients that make our argument work. One is the use of multilinear Kakeya-type theorems. These results are essentially about how families of transverse rectangular boxes in $\R^n$ intersect. There is a hierarchy of such results, illustrated by Theorem \ref{nt3}. In the earlier work \cite{BD4}, only the weakest result in this hierarchy was used (a variant of the $k=1$ case from Theorem \ref{nt3}). One novelty in the present paper is the fact that we manage to bring to bear the more complex results in the hierarchy. These  include the one where the boxes are thin tubes, a landmark result due  to Bennett, Carbery and Tao \cite{BCT}. These theorems become available to us once we set up a multilinear version of the decoupling inequality we need to prove. We make use of these multilinear Kakeya-type results  to derive the key new inequality \eqref{hdgfhgdsfdshfjhdkghfdgkjfkjhgkl} in Theorem \ref{nt4}. We call this a {\em ball inflation}. This process enlarges the size of the spatial balls, thus facilitating a subsequent decoupling into smaller intervals.

\bigskip

A second major ingredient is a form of the induction on scales from \cite{BG}, which will allow us to establish that the linear decoupling is essentially equivalent with its multilinear version. This observation will allow for a certain bootstrapping argument to gradually force the decoupling constants to get closer and closer to their conjectured values.
\bigskip

The third ingredient is an iteration scheme, whose end result is the multi-scale inequality in Theorem \ref{nt27}. This scheme builds on its earlier incarnation from \cite{BD4}, but is significantly more complex. In addition to using $L^2$ decoupling as in our previous related papers on curves, we now employ two new tools: lower dimensional decoupling and ball inflation. These will expose new features of the curve $\Gamma_n$ at appropriate scales. Each scale corresponds to a particular result from the hierarchy in Theorem \ref{nt3}.
\bigskip

We have put some effort into making our paper accessible to a large audience, one that is not necessarily familiar with the subtleties of multilinear harmonic analysis. But we also believe that  familiarity with some of the previous papers on decouplings, especially \cite{BD3}, will help the reader build some gradual understanding of our method.

\bigskip

\section{From decouplings to exponential sums}
\label{decexp}

We start with the following discrete restriction estimate which follows quite easily from our Theorem \ref{ntmain}.

\begin{theorem}
For each $1\le i\le N$, let $t_i$ be a point in $(\frac{i-1}{N},\frac{i}{N}]$. Then for each $R\gtrsim N^{n}$, each ball $B_R$ in $\R^n$, each $a_{i}\in\C$ and each $p\ge 2$  we have
$$(\frac1{|B_R|}\int|\sum_{i=1}^N a_{i}e(x_1t_i+x_2t_i^2+\ldots+x_nt_i^n)|^{p}w_{B_R}(x)dx_1\ldots dx_n)^{\frac1p}\lesssim$$
\begin{equation}
\label{fek19}
\lesssim_{\epsilon} (N^{\epsilon}+N^{\frac12(1-\frac{n(n+1)}{p})+\epsilon})\|a_{i}\|_{l^2(\{1,\ldots,N\})},
\end{equation}
and the implicit constant does not depend on $N$, $R$ and $a_{i}$.
\end{theorem}
\begin{proof}
By invoking H\"older and the trivial bound for  $L^\infty$, it suffices to prove the result for $p=n(n+1)$.
Let $\B$ be a finitely overlapping cover of $B_R$ with balls $B_{N^n}$. An elementary computation shows that
\begin{equation}
\label{fek20}
\sum_{B_{N^n}\in\B}w_{B_{N^n}}\lesssim w_{B_R},
\end{equation}
with the implicit constant independent of $N,R$. Invoking  Theorem \ref{ntmain} for each $B_{N^n}\in\B$, then summing up and using
\eqref{fek20} we obtain
$$
\|E_{[0,1]}g\|_{L^{n(n+1)}(w_{B_R})}\lesssim $$
$$
\lesssim_\epsilon N^\epsilon(\sum_{\Delta\subset [0,1]\atop{l(\Delta)=N^{-1}}}\|E_\Delta g\|_{L^{n(n+1)}(w_{B_{R}})}^2)^{1/2}.
$$
Apply this inequality to $$g=\frac1{2\tau}\sum_{i=1}^Na_{i}1_{(t_i-\tau,t_i+\tau)},$$
then let $\tau$ go to $0$.

\end{proof}
For each $1\le i\le N$ consider some real  numbers $i-1< {X}_i\le i$. We do not insist that ${X}_i$ be integers. Let $S_X=\{{X}_1,\ldots,{X}_N\}$. For each $s\ge 1$, denote by  ${J}_{s,n}(S_X)$ the number of  solutions  of the following system of inequalities
$$|X_1^i+\ldots+X_s^i-(X_{s+1}^i+\ldots+X_{2s}^i)|\le N^{i-n},\;\;1\le i\le n$$
with $X_i\in S_X$.

We can now prove the following generalization of Theorem \ref{ntmain1}.

\begin{corollary}
\label{cfek4}
For each integer $s\ge 1$ and each $S_X$ as above we have that
$${J}_{s,n}(S_X)\lesssim_\epsilon N^{s+\epsilon}+N^{2s-\frac{n(n+1)}{2}+\epsilon},$$
where the implicit constant does not depend on $S_X$.
\end{corollary}
\begin{proof}Let $\phi:\R^n\to [0,\infty)$ be a positive Schwartz function with  positive Fourier transform  satisfying $\widehat{\phi}(\xi)\ge1$ for $|\xi|\lesssim 1$.
Define $\phi_{N}(x)=\phi(\frac{x}N)$. Using the Schwartz decay, \eqref{fek19} with $a_{i}=1$ implies that for each $s\ge 1$
$$(\frac1{|B_{N^n}|}\int_{\R^n}\phi_{N^n}(x)|\sum_{i=1}^Ne(x_1t_i+\ldots+x_nt_i^n)|^{2s}dx_1\ldots dx_n)^{\frac1{2s}}$$
\begin{equation}
\label{fek21}
\lesssim_\epsilon N^{\frac12+\epsilon}+N^{1-\frac{n(n+1)}{4s}+\epsilon},
\end{equation}
whenever $t_i\in [\frac{i-1}{N},\frac{i}{N})$. Apply \eqref{fek21} to $t_i=\frac{X_i}{N}$.
Let now $$\phi_{N, 1}(x_1,x_2,\ldots,x_n)=\phi(\frac{x_1}{N^{n-1}},\frac{x_2}{N^{n-2}},\ldots,{x_n}).$$

After making a change of variables and expanding the product,  the term
$$\int_{\R^n}\phi_{N^n}(x)|\sum_{i=1}^Ne(x_1t_i+\ldots+x_nt_i^n)|^{2s}dx_1\ldots dx_n$$
can be written as  the sum over all $X_i\in S_X$ of
$$N^{\frac{n(n+1)}{2}}\int_{\R^n}\phi_{N, 1}(x)e(x_1Z_1+\ldots+x_nZ_n)dx_1\ldots dx_n,$$
where
$$Z_i=X_1^i+\ldots+X_s^i-(X_{s+1}^i+\ldots+X_{2s}^i).$$
Each such term is equal to
$$N^{n^2}\widehat{\phi}(N^{n-1}Z_1,N^{n-2}Z_2,\ldots,Z_n).$$
Recall that this is always positive, and in fact greater than $N^{n^2}$ at least ${J}_{s,n}(S_X)$ times. It now suffices to use \eqref{fek21}.
\end{proof}

\bigskip

\section{The $\epsilon$ removal argument}
For $\x=(x_1,\ldots,x_n)$ let
$$F(\x;N)=\sum_{j=1}^Ne(x_1j+x_2j^2+\ldots+ x_n j^n).$$
Write $L=N^{\frac1{2n}}$. For $1\le q\le L$, $1\le a_j\le q$ ($1\le j\le n$) with $(q,a_1,...,a_n)=1$, define the major arc
$$\M(q,a)=\{\x\in[0,1)^n:|x_j-a_j/q|\le LN^{-j}\;(1\le j\le n)\}.$$
Let $\M$ be the union of all major arcs and write $\m=[0,1)^n\setminus\M$ for the minor arcs. We recall the following two estimates from Section 7 in \cite{Wool1}
\begin{equation}
\label{hghgjj1}
\sup_{\x\in\m}|F(\x;N)|\lesssim N^{\beta},\text{ for some }\beta<1
\end{equation}
and (Lemma 7.1)
\begin{equation}
\label{hghgjj2}
\int_{\M}|F(\x;N)|^{p}d\x\lesssim N^{p-\frac{n(n+1)}{2}}, \text{ for }p>n(n+1).
\end{equation}
\bigskip
Fix now $p>n(n+1)$. On the minor arcs we can write using \eqref{hghgjj1} and Theorem \ref{ntmain1}
$$\int_{\m}|F(\x;N)|^{p}d\x\lesssim N^{\beta(p-n(n+1))}\int_{\m}|F(\x;N)|^{n(n+1)}d\x\lesssim_\epsilon$$
$$N^{\beta(p-n(n+1))}N^{\frac{n(n+1)}{2}+\epsilon}\lesssim N^{p-\frac{n(n+1)}2}.$$
Combining this with \eqref{hghgjj2} we get
\begin{equation}
\label{hghgjj3}
\int_{[0,1]^n}|F(\x;N)|^{p}d\x\lesssim  N^{p-\frac{n(n+1)}2}.
\end{equation}

\section{Transition to larger balls}
\label{nS1}

Our goal in this section is to prove Theorem \ref{nt4}. This will be the main tool that governs the ball inflation process that we use in the next section.

For $1\le j\le m$ and $1\le k\le n-1$, let $V_j$ be $k-$dimensional linear subspaces of $\R^n$ and let $\pi_j:\R^n\to V_j$ be the associated orthogonal projections.  Each $V_j$ will be equipped with the $k-$dimensional Lebesgue measure. We recall the following theorem from \cite{BCCT}.
\begin{theorem}
\label{BCCT1}
The quantity
$$\sup_{g_j\in L^{2}({V_j})\atop{g_j\not\equiv 0}}\frac{\|(\prod_{j=1}^{m}|g_j\circ \pi_j|)^{\frac1{m}}\|_{L^{\frac{2n}{k}}(\R^n)}}{(\prod_{j=1}^{m}\|g_j\|_{L^{2}(V_j)})^{\frac1{m}}}$$
is finite if and only if the following transversality requirement holds
\begin{equation}
\label{fek3}
\dim(V)\le \frac{n}{km}\sum_{j=1}^{m}{\dim(\pi_j(V))}, \text{ for every linear subspace }V\subset \R^n.
\end{equation}
\end{theorem}
We will be interested in the following consequence.

\begin{proposition}
\label{nc1}
Consider $m$ families $\P_j$ consisting of rectangular boxes $P$ in $\R^n$, which we will refer to as plates,  having the following properties
\medskip

(i) $k$ of the axes of each plate $P\in\P_j$ have side lengths equal to $R^{1/2}$ and span  $V_j$, while the remaining $n-k$ axes have side lengths  equal to $R$ (and are orthogonal to $V_j$)
\medskip

(ii) we allow each $P$ to appear multiple times within a family $\P_j$
\medskip

(iii) all plates are subsets of a ball $B_{4R}$ of radius $4R$.
\medskip

Then, assuming \eqref{fek3} holds,  we have the following inequality
\begin{equation}
\label{ne1}
\nint_{B_{4R}}|\prod_{j=1}^mF_j|^{\frac1{2m}\frac{2n}{k}}\le C(V_1,\ldots,V_m) \left[\prod_{j=1}^m|\nint_{B_{4R}}F_j|^{\frac1{2m}}\right]^{\frac{2n}k}
\end{equation}
for all functions  $F_j$ of the form
$$F_j=\sum_{P\in\P_j}c_P1_P.$$
The finite quantity $C(V_1,\ldots,V_m)$ will not depend on $R,c_P,\P_j$, but will depend on $V_j$.
\end{proposition}
\begin{proof}
Since we allow multiple repetitions, it suffices to prove \eqref{ne1} under the assumption that $c_P=1$ for each $P$. Indeed, the case of rational $c_P$ is then  immediate, while the case of arbitrary $c_P\in\C$ follows by density arguments. Let $N_j$ be the number of plates in $\P_j$. We need to prove that
$$\int_{B_{4R}}(\prod_{j=1}^m\sum_{P\in\P_j}1_P)^{\frac{n}{mk}}\lesssim R^{n/2}\prod_{j=1}^mN_j^{\frac{n}{mk}}.$$

For each $P\in\P_j$, there exists $v_P\in V_j$ such that
$$P\subset\{x\in\R^n:\;\pi_j(x)\in Q_P\},$$
where $Q_P$ is a cube in $V_j$ centered at $v_P$, with side length $R^{1/2}$.
Apply Theorem \ref{BCCT1} to the functions
$$g_j=(\sum_{P\in\P_j}1_{Q_P})^{1/2}.$$
It suffices to note that $$g_j\circ\pi_j(x)\ge (\sum_{P\in\P_j}1_P)^{1/2}(x),\;\;x\in\R^n$$
and that $$\|g_j\|_{L^2(V_j)}\sim N_j^{\frac12}R^{k/4}.$$
\end{proof}
\bigskip

For $1\le k\le n-1$, we will consider the following $k-$dimensional linear spaces  in $\R^n$
$$V_k(t)=\langle \Phi'(t),\Phi''(t),\ldots,\Phi^{(k)}(t)\rangle,\;\; t\in[0,1].$$
Recall that $\Phi$ gives the parametric representation \eqref{ryvbt78y783956t787034r894} of the curve $\Gamma$.
\begin{lemma}
\label{nl1}
Let $M_n=n!$. Then for each $M_n$ pairwise distinct  points $t_j$ in $[0,1]$, the spaces $(V_k(t_j))_{j=1}^{M_n}$ satisfy the non-degeneracy condition \eqref{fek3} with $m=M_n$.
\end{lemma}
\begin{proof}
Let $\pi_j$ be the orthogonal projection associated with $V_k(t_j)$.
Fix a linear space $V$ in $\R^n$, and denote by $\gamma_1$ the orthogonal projection of $\Gamma$ onto $V$. Then
$$\gamma_1(t)=(P_1(t),\ldots,P_n(t)),$$
for some polynomials $P_j$ of degree at most $n$. Write
$$\Phi(t)=\gamma_1(t)+\gamma_2(t)$$
and note that for each $t$ $$\gamma_2^{(i)}(t)\perp V,\;\;i\ge 0.$$
Inequality \eqref{fek3} is trivial when $\dim(V)=n$, so we may assume that $\dim(V)\le n-1$.

Let us start with the case $\dim(V)\le k$. We will show that
$$\dim(\pi_j(V))<\dim(V)$$
can only happen for at most $(n-1)!$ values of $j$.

Pick $j$ so that $\dim(\pi_j(V))<\dim(V)$. Then there must exist $v_j\in V$ orthogonal to $V_k(t_j)$. It follows that $v_j$ is orthogonal to the vectors $\gamma_1^{(i)}(t_j)$, $1\le i\le k$. Since $\gamma_1^{(i)}(t_j)\in V$, we infer that
$$\gamma_1^{(1)}(t_j)\wedge\ldots\wedge \gamma_1^{(\dim(V))}(t_j)=0.$$
Assume for contradiction that this held true for at least $(n-1)!+1$ points $t_j$. By the Fundamental Theorem of Algebra, this would force
$$\gamma_1^{(1)}(t)\wedge\ldots\wedge \gamma_1^{(\dim(V))}(t)\equiv 0.$$
Using a Wronskian argument, it further follows that the whole curve $\gamma_1$ lives in a linear subspace of $V$ with dimension $\dim(V)-1$. This leads to the contradiction that $\Gamma$ lives in an $n-1$ dimensional subspace of $\R^n$. It suffices now to note that
$$\frac{n}{kM_n}\sum_{j=1}^{M_n}{\dim(\pi_j(V))}\ge \dim(V)\frac{n}{n-1}\frac{n!-(n-1)!}{n!}=\dim(V).$$
\medskip

When $k\le \dim(V)\le n-1$, the argument above shows that
$$\dim(\pi_j(V))=k$$
for all but at most $(n-1)!$ values of $j$. We can write now
$$\frac{n}{kM_n}\sum_{j=1}^{M_n}{\dim(\pi_j(V))}\ge n\frac{n!-(n-1)!}{n!}=n-1\ge \dim(V).$$

\end{proof}
\bigskip
We will need the following uniform version of Proposition \ref{nc1}.
\begin{proposition}
\label{np2}
For each $K\ge M_n$ we have
$$\sup_{(t_1,\ldots,t_{M_n})\in S_K}C(V_k(t_1),\ldots,V_k(t_{M_n}))<\infty,$$
where
$$S_K=\{(t_1,\ldots,t_{M_n})\in[0,1]^{M_n}:\;|t_i-t_j|\ge \frac1K,\text{ for all }1\le i\not=j\le M_n\}.$$
\end{proposition}
\begin{proof}
Let $\pi_{V_k(t)}$ be the orthogonal projection associated with $V_k(t)$.
The proof of Proposition \ref{nc1} shows that it  suffices to prove that
$$\sup_{(t_1,\ldots,t_{M_n})\in S_K}\sup_{g_j\in L^{2}({V_k(t_j)})}\frac{\|(\prod_{j=1}^{M_n}|g_j\circ \pi_{V_k(t_j)}|)^{\frac1{M_n}}\|_{L^{\frac{2n}{k}}(\R^n)}}{(\prod_{j=1}^{M_n}\|g_j\|_{L^{2}(V_k(t_j))})^{\frac1{M_n}}}<\infty.$$
The function
$$H(t_1,\ldots,t_{M_n})=\sup_{g_j\in L^{2}({V_k(t_j)})}\frac{\|(\prod_{j=1}^{M_n}|g_j\circ \pi_{V_k(t_j)}|)^{\frac1{M_n}}\|_{L^{\frac{2n}{k}}(\R^n)}}{(\prod_{j=1}^{M_n}\|g_j\|_{L^{2}(V_k(t_j))})^{\frac1{M_n}}}$$
is finite on $S_K$ due to Theorem \ref{BCCT1} and Lemma \ref{nl1}. Moreover, Theorem 1.1 in \cite{BBFL} proves that if $H(t_1,\ldots,t_{M_n})<\infty$ for some $(t_1,\ldots,t_{M_n})\in[0,1]^{M_n}$, then $$\sup H(t_1',\ldots,t_{M_n}')<\infty$$ on some neighborhood of $(t_1,\ldots,t_{M_n})$. It now suffices to note that $S_K$ is compact.

\end{proof}

We can now extend the result of Proposition \ref{nc1} to get a multilinear Kakeya-type inequality for plates. For some $K\ge 2M_n$, let $I_1,\ldots,I_{M_n}$ be intervals of length $\frac1K$ in $[0,1]$, which in addition are assumed to be separated by at least $\frac1K$.
\begin{theorem}
\label{nt3}
Consider $M_n$ families $\P_j$ consisting of plates $P$ in $\R^n$ having the following properties
\medskip

(i) for each $P\in\P_j$ there exists $t_P\in I_j$ such that $k$ of the axes of $P$ have side lengths equal to $R^{1/2}$ and span  $V_k(t_P)$, while the remaining $n-k$ axes have side lengths  equal to $R$ (we will refer to the orientation of such a $P$ as being $(V_k(t_P), V_k(t_P)^{\perp})$)
\medskip

(ii) all plates are subsets of a ball $B_{4R}$ of radius $4R$
\medskip

Then we have the following inequality
\begin{equation}
\label{ne3}
\nint_{B_{4R}}|\prod_{j=1}^{M_n}F_j|^{\frac1{2M_n}\frac{2n}{k}}\lesssim_{\epsilon, K}R^{\epsilon} \left[\prod_{j=1}^{M_n}|\nint_{B_{4R}}F_j|^{\frac1{2M_n}}\right]^{\frac{2n}k}
\end{equation}
for all functions  $F_j$ of the form
$$F_j=\sum_{P\in\P_j}c_P1_P.$$
The implicit constant will not depend on $R,c_P,\P_j$.
\end{theorem}
\begin{proof}
There is a rather short proof of the case $k=n-1$ in \cite{Gu}, that easily extends to the other values of $k$. See also \cite{BD6} for a further discussion.
\end{proof}

\bigskip

Let  $\rho<\frac1K$. We close this section by proving the key new inequality in this paper, that shows how to pass from balls of smaller radius to balls of larger radius. Note that \eqref{hdgfhgdsfdshfjhdkghfdgkjfkjhgkl} is not a decoupling inequality, as the intervals $J_i$ remain unchanged from left to right. However, the ball on the right is larger, which will allow for  subsequent decouplings into smaller intervals (via both $L^2$ and lower dimensional decouplings). The general underlying principle is that the larger the ball, the smaller the decoupling intervals are.

\begin{theorem}
\label{nt4}
Fix $1\le k\le n-1$ and $p\ge 2n$.
Let $B$ be an arbitrary ball in $\R^n$ with radius $\rho^{-(k+1)}$, and let $\B$ be a finitely overlapping cover of $B$ with balls $\Delta$ of radius $\rho^{-k}$. Then for each $g:[0,1]\to \C$ we have
\begin{equation}
\label{hdgfhgdsfdshfjhdkghfdgkjfkjhgkl}
\frac1{|\B|}\sum_{\Delta\in\B}\left[\prod_{i=1}^{M_n}(\sum_{J_i\subset I_i\atop{|J_i|=\rho}}\|E_{J_i}g\|_{L_{\sharp}^{\frac{pk}{n}}(w_\Delta)}^2)^{1/2}\right]^{p/{M_n}}\lesssim_{\epsilon,K}\rho^{-\epsilon} \left[\prod_{i=1}^{M_n}(\sum_{J_i\subset I_i\atop{|J_i|=\rho}}\|E_{J_i}g\|_{L_{\sharp}^{\frac{pk}{n}}(w_B)}^2)^{1/2}\right]^{p/M_n}.
\end{equation}
Moreover, the implicit constant is independent of $g,\rho,B$.
\end{theorem}
\begin{proof}
Since we can afford logarithmic losses in $\rho$, it suffices to prove the inequality with the summation on both sides restricted to  families of $J_i$ for which  $\|E_{J_i}g\|_{L_{\sharp}^{\frac{pk}{n}}(w_B)}$ have comparable size (up to a factor of 2), for each $i$. Indeed, the $J_{i'}$ satisfying
$$\|E_{J_{i'}}g\|_{L_{\sharp}^{\frac{pk}{n}}(w_B)}\le \rho^C\max_{i}\|E_{J_{i}}g\|_{L_{\sharp}^{\frac{pk}{n}}(w_B)}$$
can be easily dealt with by using the triangle inequality, since we automatically have
$$\max_{\Delta\in \B}\|E_{J_{i'}}g\|_{L_{\sharp}^{\frac{pk}{n}}(w_\Delta)}\le \rho^C\max_{i}\|E_{J_{i}}g\|_{L_{\sharp}^{\frac{pk}{n}}(w_B)}.$$
This leaves only $\log_2 (\rho^{-O(1)})$ sizes to consider.
\medskip

Let us now assume that we have $N_i$ intervals $J_i$, with $\|E_{J_{i}}g\|_{L_{\sharp}^{\frac{pk}{n}}(w_B)}$ of comparable size. Since $p\ge 2n$, by H\"older's inequality the left hand side  is  at most
\begin{equation}
\label{ne4}
(\prod_{i=1}^{M_n}N_i^{\frac12-\frac{n}{pk}})^{p/{M_n}}\frac1{|\B|}\sum_{\Delta\in\B}\left[\prod_{i=1}^{M_n}(\sum_{J_i}\|E_{J_i}g\|_{L_{\sharp}^{\frac{pk}{n}}(w_\Delta)}^{\frac{pk}{n}})\right]^{\frac{n}{k M_n}}.
\end{equation}

\medskip

Fix $B$,  a ball of radius $\rho^{-k-1}$, and a cover $\B$ by balls $\Delta$.  For each interval $J=J_i$ of the form $[t_J-\rho/2,t_J+\rho/2]$ we cover $\cup_{\Delta\in\B}\Delta$ with a family $\F_J$ of  tiles $ T_J$ with orientation $(V_k(t_J), V_k(t_J)^{\perp})$,  $k$ short sides of length $\rho^{-k}$ and $n-k$ longer sides of length $\rho^{-k-1}$. Moreover, we can assume these tiles to be inside the ball $4B$. We let $ T_J(x)$ be the  tile containing $x$, and we let $2  T_J$ be the dilation of $ T_J$ by a factor of 2 around its center.

Let us use $q$ to abbreviate $pk/n$, and $\alpha$ to abbreviate $\frac{n}{k M_n}$.     Our goal is to control the expression

$$ \Avg_{\Delta\in\B} \prod_i \left( \sum_{J_i} \| E_{J_i} g \|_{L^q_{\sharp}(w_\Delta)}^q \right)^\alpha. $$

We now define $F_J$ for $x\in \cup_{ T_J\in\F_J} T_J$ by

$$ F_J(x) := \sup_{y \in 2  T_J(x)} \| E_J g \|_{L^q_{\sharp}(w_{B(y, \rho^{-k})})}. $$

For any point $x \in \Delta$ we have $ \Delta \subset 2  T_J(x)$, and so we also have

$$ \| E_J g \|_{L^q_{\sharp}(w_\Delta)} \le F_J (x). $$

Therefore,

$$ \Avg_{\Delta \in \B} \prod_i \left( \sum_{J_i} \| E_{J_i} g \|_{L^q_{\sharp}(w_\Delta)}^q \right)^\alpha \lesssim \nint_{4B} \prod_i (\sum_{J_i} F_{J_i}^q )^\alpha. $$

Moreover, the function $F_J$ is constant on each tile $ T_J\in\F_J$.  Applying Theorem 6.5 (note that $F_J$ satisfy a stronger property than what is needed), we get the bound

$$ \nint_{4B} \prod_{i} (\sum_{J_i} F_{J_i}^q )^\alpha \lesssim_{\epsilon,K} \rho^{-\epsilon} \prod_i \left( \sum_{J_i} \nint_{4B} F_{J_i}^q \right)^\alpha. $$

It remains to check that for each $J=J_i$

\begin{equation} \label{FJbound} \| F_J \|_{L^q_{\sharp}(4B)} \lesssim \| E_J g \|_{L^q_{\sharp}(w_B)}. \end{equation}
Once this is established, it follows  that \eqref{ne4} is dominated by
\begin{equation}
\label{ne5}
\rho^{-\epsilon}(\prod_{i=1}^{M_n}N_i^{\frac12-\frac{n}{pk}})^{p/{M_n}}\prod_{i=1}^{M_n}(\sum_{J_i}\|E_{J_i}g\|_{L_{\sharp}^{\frac{pk}{n}}(w_B)}^{\frac{pk}n})^{\frac{n}{kM_n}}.
\end{equation}
Recalling the restriction we have made on $J_i$, \eqref{ne5} is comparable to
$$\rho^{-\epsilon}\left[\prod_{i=1}^{M_n}(\sum_{J_i}\|E_{J_i}g\|_{L_{\sharp}^{\frac{pk}{n}}(w_B)}^2)^{1/2}\right]^{p/M_n},$$
as desired.
\medskip

To prove \eqref{FJbound}, we may assume $J=[-\rho/2,\rho/2]$ and thus $\widehat{E_Jg}$ will be supported in $\prod_{j=1}^n[-\rho^j,\rho^j]$. Fix $x=(x_1,\ldots,x_n)$ with $ T_J(x)\in\F_J$ and $y\in 2 T_J(x)$. Note that $T_J(x)$ has sides parallel to the coordinate axes. In particular, $y=x+y'$ with $|y_j'|<4\rho^{-k}$ for $1\le j\le k$ and $|y_j'|<4\rho^{-k-1}$, for $k+1\le j\le n$. Then
\begin{equation}
\label{h dfhivuhy uytugr9495m8t9349r-=349t-5=9t-0=}
\| E_J g \|_{L^q(w_{B(y, \rho^{-k})})}^q\lesssim
\end{equation}
$$\int|E_Jg(x_1+u_1,\ldots,x_k+u_k,x_{k+1}+u_{k+1}+y_{k+1}',\ldots,x_n+u_n+y_n')|^qw_{B(0, \rho^{-k})}(u)du.$$
Now, using Taylor expansion we get
$$|E_Jg(x_1+u_1,\ldots,x_k+u_k,x_{k+1}+u_{k+1}+y_{k+1}',\ldots,x_n+u_n+y_n')|$$
$$=|\int\widehat{E_Jg}(\lambda)e(\lambda\cdot(x+u))e(\lambda_{k+1}y_{k+1}')\ldots e(\lambda_{n}y_{n}')d\lambda|\le$$
$$\le \sum_{s_{k+1}=0}^{\infty}\frac{100^{s_{k+1}}}{s_{k+1}!}\ldots \sum_{s_{n}=0}^{\infty}\frac{100^{s_{n}}}{s_{n}!}|\int\widehat{E_Jg}(\lambda)e(\lambda\cdot(x+u))(\frac{\lambda_{k+1}}{2\rho^{k+1}})^{s_{k+1}}\ldots(\frac{\lambda_{n}}{2\rho^{{k+1}}})^{s_{n}}d\lambda|$$
$$=\sum_{s_{k+1}=0}^{\infty}\frac{100^{s_{k+1}}}{s_{k+1}!}\ldots \sum_{s_{n}=0}^{\infty}\frac{100^{s_{n}}}{s_{n}!} |M_{s_{k+1},\ldots,s_n}(E_Jg)(x+u)|.$$
Here $M_{s_{k+1},\ldots,s_n}$ is the operator with Fourier multiplier
$m_{s_{k+1},\ldots,s_n}(\frac{\lambda}{2\rho^{k+1}})$, where
$$m_{s_{k+1},\ldots,s_n}(\lambda)=({\lambda_{k+1}})^{s_{k+1}}\ldots({\lambda_{n}})^{s_{n}}1_{[-1/2,1/2]}(\lambda_{k+1})\ldots1_{[-1/2,1/2]}(\lambda_{n}).$$
We are able to insert the cutoff because of our initial restriction on the Fourier support of $E_Jg$.

Plugging this estimate into \eqref{h dfhivuhy uytugr9495m8t9349r-=349t-5=9t-0=} we obtain
$$\| E_J g \|_{L^q_{\sharp}(w_{B(y, \rho^{-k})})}^q\lesssim \sum_{s_{k+1}=0}^{\infty}\frac{100^{s_{k+1}}}{s_{k+1}!}\ldots \sum_{s_{n}=0}^{\infty}\frac{100^{s_{n}}}{s_{n}!}\|M_{s_{k+1},\ldots,s_n}(E_Jg)\|_{L_\sharp^q(w(x,\rho^{-k}))}^q.$$
Recalling the definition of $F_J$ and the fact that
$$\int_{4B}w_{B(x,\rho^{-k})}(z)dx\lesssim w_B(z),\;z\in\R^n$$
we conclude that
\begin{equation}
\label{hgefcgeryiuftrebyubfureywuyui}
\| F_J \|_{L^q_{\sharp}(4B)}^q\lesssim \sum_{s_{k+1}=0}^{\infty}\frac{100^{s_{k+1}}}{s_{k+1}!}\ldots \sum_{s_{n}=0}^{\infty}\frac{100^{s_{n}}}{s_{n}!}\|M_{s_{k+1,\ldots,s_n}}(E_Jg)\|_{L_\sharp^q(w_B)}^q.
\end{equation}

Note that $m_{s_{k+1},\ldots,s_n}$ admits a smooth extension to the whole $\R^n$ (and we will use the same notation for it) with derivatives of any given order uniformly bounded over $s_{k+1},\ldots,s_n$. It follows that
$$|\widehat{m}_{s_{k+1},\ldots,s_n}(x)|\lesssim\xi(x_{k+1},\ldots,x_n),$$
with implicit constants independent of $s_{k+1},\ldots,s_n$, where
$$\xi(x_{k+1},\ldots,x_n)\lesssim_M (1+(x_{k+1}^2+\ldots+x_n^2)^{1/2})^{-M},$$
for all $M>0$.
 We can now write
$$|M_{s_{k+1},\ldots,s_n}(E_Jg)(x)|\lesssim |E_Jg|\odot\xi_{\rho^{k+1}}$$
where $\odot$ denotes the convolution  with respect to the last $n-k$ variables $\bar x$, and
$$\xi_{\rho^{k+1}}(\bar x)=\rho^{(n-k)(k+1)}\xi(\rho^{k+1}\bar x).$$
Using this, one can easily check that
$$\|M_{s_{k+1,\ldots,s_n}}(E_Jg)\|_{L^q(w_B)}^q\lesssim \langle |E_Jg|^q\odot\xi_{\rho^{k+1}},w_B\rangle$$$$ =\langle |E_Jg|^q,\xi_{\rho^{k+1}}\odot w_B\rangle\lesssim \langle |E_Jg|^q,w_B\rangle.$$
Combining this with \eqref{hgefcgeryiuftrebyubfureywuyui} leads to the proof of \eqref{FJbound}
$$\| F_J \|_{L^q_{\sharp}(4B)}^q\lesssim \sum_{s_{k+1}=0}^{\infty}\frac{100^{s_{k+1}}}{s_{k+1}!}\ldots \sum_{s_{n}=0}^{\infty}\frac{100^{s_{n}}}{s_{n}!}\|E_Jg\|_{L_\sharp^q(w_B)}^q$$
$$\lesssim \|E_Jg\|_{L_\sharp^q(w_B)}^q.$$
The argument is now complete.

\end{proof}
\bigskip

For future use, we will record here the following easy analogue of Theorem \ref{nt4} for $k=n$. Let $B$ be an arbitrary ball in $\R^n$ with radius $R$, and let $\B$ be a finitely overlapping cover of $B$ with balls $\Delta$ of radius $r<\frac{R}{2}$. No other relation between $\rho,r,R$ is assumed.  Then for each $g:[0,1]\to \C$ and $p\ge 2$  we have
\begin{equation}
\label{j nbyrtug8rt7g856uy956y9}
\frac1{|\B|}\sum_{\Delta\in\B}\left[\prod_{i=1}^{M_n}(\sum_{J_i\subset I_i\atop{|J_i|=\rho}}\|E_{J_i}g\|_{L_{\sharp}^{p}(w_\Delta)}^2)^{1/2}\right]^{p/{M_n}}\lesssim\left[\prod_{i=1}^{M_n}(\sum_{J_i\subset I_i\atop{|J_i|=\rho}}\|E_{J_i}g\|_{L_{\sharp}^{p}(w_B)}^2)^{1/2}\right]^{p/M_n}.
\end{equation}
The proof relies on the inequality
$$\sum_{\Delta\in\B}w_\Delta\lesssim w_B,$$
on H\"older's inequality, and on the following consequence of Minkowski's inequality
$$\sum_{i}(\sum_k\|f_k\|_{L^p(w_i)}^2)^{\frac{p}2}\le (\sum_k\|f_k\|_{L^p(\sum_iw_i)}^2)^{\frac{p}2},$$
valid for arbitrary $f_k:\R^n\to\C$ and arbitrary positive functions $w_i:\R^n\to [0,\infty)$.

\bigskip

\section{A few basic tools}

Fix a positive Schwartz function $\eta:\R^n\to[0,\infty)$. For each ball $B=B(c,R)$ in $\R^n$, define
$$\eta_B(x)=\eta(\frac{x-c}{R}).$$
\begin{lemma}
\label{nl9}
Let $\W$ be the collection of all weights, that is, positive, integrable functions on $\R^n$. Fix a radius $R>0$.
Let $Q_1,Q_2:\W\to[0,\infty]$ have the following four properties:

(1) $Q_1(1_B)\lesssim Q_2(\eta_B)$ for all balls $B\subset \R^n$ of radius $R$

(2) $Q_1(u+v)\le Q_1(u)+Q_1(v)$, for each $u,v\in\W$

(3) $Q_2(u+v)\ge Q_2(u)+Q_2(v)$, for each $u,v\in\W$

(4) If $u\le v$ then $Q_i(u)\le Q_i(v)$.

Then $$Q_1(w_B)\lesssim Q_2(w_B)$$
for each ball $B$ with radius $R$. The implicit constant is independent of $R$.
\end{lemma}
\begin{proof}
Let $\B$ be a finitely overlapping cover of $\R^n$ with balls $B'=B'(c_{B'},R)$.
It suffices to note that
$$w_B(x)\lesssim \sum_{B'\in\B}1_{B'}(x)w_B(c_{B'})$$
and that
$$\sum_{B'\in\B}\eta_{B'}(x)w_B(c_{B'})\lesssim w_B(x).$$
\end{proof}
Here is one useful consequence of the lemma.
\begin{corollary}
For each $q\ge p\ge 1$ and each ball $B$ in $\R^n$ with radius at least 1 we have
\begin{equation}
\label{ nghbugtrt90g0-er9t-9}
\|E_{[0,1]}g\|_{L^q(w_B)}\lesssim \|E_{[0,1]}g\|_{L^p(w_B)},
\end{equation}
with the implicit constant independent of $B$ and $g$.
\end{corollary}
\begin{proof}
Let $\eta$ be a positive smooth function on $\R^n$  satisfying $1_B\le \eta_B$ and such that the Fourier transform of $\eta^{\frac1p}$ is  supported in $B(0,1)$.
We can thus write
$$\|E_{[0,1]}g\|_{L^q(B)}\lesssim \|E_{[0,1]}g\|_{L^q(\eta_B^{\frac{q}{p}})}=\|\eta_B^{\frac1p}E_{[0,1]}g\|_{L^q(\R^n)}.$$
Since the Fourier transform of $\eta_B^{\frac1p}E_{[0,1]}g$ is supported in the ball $B(0,3)$, we have
$$\|\eta_B^{\frac1p}E_{[0,1]}g\|_{L^q(\R^n)}\lesssim \|\eta_B^{\frac1p}E_{[0,1]}g\|_{L^p(\R^n)}=\|E_{[0,1]}g\|_{L^p(\eta_B)}.$$
Apply now Lemma \ref{nl9} with $Q_1(u)=\|E_{[0,1]}g\|_{L^q(u)}^q$ and $Q_2(u)=\|E_{[0,1]}g\|_{L^p(u)}^q$.

\end{proof}

\bigskip

We continue by  recalling a few basic tools and inequalities from \cite{BD4}, that will be used in our iteration scheme.

Let $n\ge 2$, $p\ge 2$, $0<\delta\le 1$. We will denote by $V_p(\delta)=V_{p,n}(\delta)$ the smallest constant such that the inequality
$$\|E_{[0,1]}g\|_{L^p(w_B)}\le V_p(\delta)(\sum_{J\subset [0,1]\atop{|J|=\delta}}\|E_Jg\|_{L^p(w_B)}^2)^{1/2}$$
holds true for each ball $B\subset \R^n$ with radius $\delta^{-n}$ and each $g:[0,1]\to\C$.
\bigskip

We will rely on the following generalization of parabolic rescaling. See for example Section 7 in \cite{BD6} for a proof.
\begin{lemma}
Let $0<\rho\le 1$.
For each interval $I\subset [0,1]$ with length $\delta^{\rho}$  and each ball $B\subset \R^n$ with radius $\delta^{-n}$ we have
\begin{equation}
\label{ne31}
\|E_Ig\|_{L^p(w_B)}\lesssim V_p(\delta^{1-\rho})(\sum_{J\subset I\atop{|J|=\delta}}\|E_Jg\|_{L^p(w_{B})}^2)^{1/2}.
\end{equation}
\end{lemma}

 Let us now introduce a multilinear version of $V_p(\delta)$. Recall the relevance of $M_n=n!$ from Section \ref{nS1}.  Given also $K\ge 2M_n$, we will denote by $V_p(\delta,K)=V_{p,n}(\delta,K)$ the smallest constant such that the inequality
$$\|(\prod_{i=1}^{M_n}E_{I_i}g)^{1/M_n}\|_{L^p(w_B)}\le V_p(\delta,K)\prod_{i=1}^{M_n}(\sum_{J\subset I_i\atop{|J|=\delta}}\|E_Jg\|_{L^p(w_B)}^2)^{\frac1{2M_n}}$$
holds true for each ball $B\subset \R^n$ with radius $\delta^{-n}$, each $g:[0,1]\to\C$ and all intervals
$I_1,\ldots,I_{M_n}$  of the form $[\frac{i}{K},\frac{i+1}{K}]$ in $[0,1]$, that in addition are assumed to be non adjacent.
\bigskip

It is immediate that $V_p(\delta,K)\le V_p(\delta)$. The reverse inequality is also essentially true, apart from negligible losses. The proof is very similar to that of Theorem 8.1 in \cite{BD6}.
\begin{theorem}
For each $K$ there exists $\epsilon_p(K)$ with $\lim_{K\to\infty}\epsilon_p(K)=0$ and $C_{K,p}$ so that
\begin{equation}
\label{ne29}
V_p(\delta)\le C_{K,p}\delta^{-\epsilon_p(K)}\sup_{\delta\lesssim \delta'<1}V_p(\delta',K),
\end{equation}
for each $0<\delta\le 1$.
\end{theorem}
It is worth mentioning that the value of $M_n$, as long as it only depends  on $n$, is irrelevant for the validity of \eqref{ne29}.
\bigskip

\section{The iteration scheme}
The goal of this section is to prove the multi-scale inequality in Theorem \ref{nt27}. There will be three different principles that we will apply repeatedly.
\medskip

$\bullet$ \textbf{$L^2$ decoupling:} This exploits $L^2$ orthogonality and will allow us to decouple to the smallest possible scale, equal to the inverse of the radius of the ball. This principle is illustrated by the following simple result.
\begin{lemma}
\label{nl3}
Let $\J_i$ be arbitrary collections of pairwise disjoint intervals $J$ with length equal to an integer multiple of  $R^{-1}$. Then for each integer $M\ge 1$ and each $B_R\subset\R^n$ we have
$$\big[\prod_{i=1}^{M}(\sum_{J\in \J_i}\|E_{J}g\|_{L^2_\sharp(w_{B_R})}^2)^{1/2}\big]^{\frac1{M}}\lesssim \big[\prod_{i=1}^{M}(\sum_{|\Delta|=R^{-1},\;\Delta\subset J\atop{\text{ for some }J\in\J_i}}\|E_{\Delta}g\|_{L^2_\sharp(w_{B_R})}^2)^{1/2}\big]^{\frac1{M}}.$$
\end{lemma}

\begin{proof}
The proof will not exploit any transversality, so the value of $M$ will not matter.
It suffices to prove that
\begin{equation}
\label{ne89}
\|E_{J}g\|_{L^2(w_{B_R})}^2\lesssim \sum_{|\Delta|=R^{-1},\;\Delta\subset J}\|E_{\Delta}g\|_{L^2(w_{B_R})}^2.
\end{equation}
Fix   a positive Schwartz function $\eta$ such that the Fourier transform of $\sqrt{\eta}$ is supported in a small neighborhood of the origin, and whose values are nonzero on the unit ball $B(0,1)$. For a ball $B=B(c,R)$ define
$$\eta_B(x)=\eta(\frac{x-c}{R}).$$
By invoking Lemma \ref{nl9} we see that inequality \eqref{ne89} will follow once we check that
\begin{equation}
\label{ne90}
\|E_{J}g\|_{L^2({B'})}^2\lesssim \sum_{|\Delta|=R^{-1},\;\Delta\subset J}\|E_{\Delta}g\|_{L^2(\eta_{B'})}^2
\end{equation}
holds true for each ball $B'$ with radius $R$.

Note that the Fourier transform of $\sqrt{\eta_{B'}}E_{J}g$ will be supported inside some $R^{-1}-$neighborhood of the arc $\Gamma_J$, and that these neighborhoods are pairwise disjoint for two non adjacent arcs. Since
$$\|E_{J}g\|_{L^2({B'})}^2\lesssim \|E_{J}g\|_{L^2(\eta_{B'})}^2=\|\sqrt{\eta_{B'}}E_{J}g\|_{L^2(\R^n)}^2,$$
\eqref{ne90} will now immediately follow from the $L^2$ orthogonality of the functions $\sqrt{\eta_{B'}}E_{J}g$.

\end{proof}
\medskip

$\bullet$ \textbf{Lower dimensional decoupling:} Various arcs on the curve $\Gamma_n$ in $\R^n$ will look lower dimensional at the appropriate scale. As a result, they will be decoupled into smaller arcs using Theorem \ref{ntmain} in lower dimensions. This is illustrated by the following result. To avoid confusion, we will use the notation $E^{(n)}$ for the extension operator relative to $\Gamma_n$ in $\R^n$.
\begin{lemma}
\label{nl5}
Let $I=[t_0,t_0+\sigma]\subset [0,1]$ and let $3\le k\le n$.
Then for each $\sigma^{-k+1}\lesssim R$, each ball $B\subset\R^n$ with radius $R$ and each $p\ge 2$ we have
$$\|E_{I}^{(n)}g\|_{L^{p}_\sharp(w_B)}\lesssim V_{p,k-1}(\sigma)(\sum_{|J|=R^{-\frac1{k-1}},\;J\subset I}\|E_{J}^{(n)}g\|_{L^{p}_\sharp(w_B)}^2)^{1/2}.$$
\end{lemma}
\begin{proof}
To simplify numerology, let us assume $k=n=3$. The proof in the general case  follows via trivial modifications. We rely on three observations. First, note that it suffices to consider the case  $R\sim\sigma^{-2}$. The general case follows by summing over balls with radius $R\sim\sigma^{-2}$.

Second, using the change of variables $s=t-t_0$ we can rewrite
$$|E_{I}^{(3)}g(x_1,x_2,x_3)|=|\int_0^\sigma g_{t_0}(s)e(s(x_1+2t_0x_2+3t_0^2x_3)+s^2(x_2+3t_0x_3)+s^3x_3)ds|,$$
$$|E_{J}^{(3)}g(x_1,x_2,x_3)|=|\int_{J-t_0} g_{t_0}(s)e(s(x_1+2t_0x_2+3t_0^2x_3)+s^2(x_2+3t_0x_3)+s^3x_3)ds|,$$
with
$$g_{t_0}(s)=g(s+t_0).$$
Since the linear transformation
$$(x_1,x_2,x_3)\mapsto(x_1+2t_0x_2+3t_0^2x_3,x_2+3t_0x_3,x_3)$$
has bounded distortion for $t_0\in[0,1]$, it suffices to prove that for each $g:[0,\sigma]\to\C$
\begin{equation}
\label{ne30}
\|E_{[0,\sigma]}^{(3)}g\|_{L^{p}_\sharp(w_B)}\lesssim V_{p,2}(R^{-\frac1{2}})(\sum_{|J|=R^{-\frac1{2}},\;J\subset [0,\sigma]}\|E_{J}^{(3)}g\|_{L^{p}_\sharp(w_B)}^2)^{1/2}.
\end{equation}
Third, we observe that since $\sigma^{3}\lesssim R^{-1}\sim \sigma^2$, the curve
$$\{(t,t^2,t^3),\;t\in[0,\sigma]\}$$
is within an $O(R^{-1})-$neighborhood of the (planar) parabola
$$(t,t^2,0),\;t\in[0,\sigma].$$

Since we decouple on balls $B$ of radius $R$, the two curves are indistinguishable at this scale and \eqref{ne30} follows via standard mollification arguments. See for example \cite{BD3} or \cite{BD4}.

\end{proof}
\medskip

$\bullet$ \textbf{Ball inflation:} We will iterate our multi-scale inequality by repeatedly passing  to balls of larger radius, a process we call {\em ball inflation}. This is encoded by Theorem \ref{nt4} and by inequality \eqref{j nbyrtug8rt7g856uy956y9}, that will work as a substitute for Theorem \ref{nt4}, when $k=n$.
The motivation for the ball inflation is that once we integrate on larger balls, we are able to decouple into intervals of finer scales, by combining the previously mentioned $L^2$ and lower dimensional decouplings.

The proof of the ball inflation inequality  \eqref{hdgfhgdsfdshfjhdkghfdgkjfkjhgkl} relies on the hierarchy of   multilinear Kakeya-type inequalities described in  Theorem \ref{nt3}. This is the place in our argument where the merit of  the multilinear perspective is revealed.
\bigskip

For the remainder of the section we fix an integer $K\ge 2M_n$ and the intervals
$I_1,\ldots,I_{M_n}$ of the form $[\frac{i}{K},\frac{i+1}{K}]$, with $i$ in a collection of non consecutive integers among $0,1,\ldots,K-1$. Note that there are finitely many such choices of intervals, so the implicit constants below (such as the one in  \eqref{ne67}) can be taken to be uniform over the intervals we are using.
\bigskip

Given a function $g:[0,1]\to\C$, a ball $B^r$ with radius $\delta^{-r}$ in $\R^n$, $t\ge 1$ and $q\le s \le r$, write
$$D_t(q,B^r)=D_t(q,B^r,g)=\big[\prod_{i=1}^{M_n}(\sum_{J_{i,q}\subset I_i\atop{|J_{i,q}|=\delta^q}}\|E_{J_{i,q}}g\|_{L^{t}_{\sharp}(w_{B^r})}^2)^{1/2}\big]^{\frac{1}{M_n}
}.$$
Also, given a finitely overlapping cover $\B_s(B^r)$ of  $B^r$ with balls $B^s$, we let
$$A_p(q,B^r,s)=A_p(q,B^r,s,g)=\big(\frac1{|\B_s(B^r)|}\sum_{B^s\in\B_s(B^r)}D_2(q,B^s)^p\big)^{\frac1p}.$$
The letter $A$ will remind us that we have an average.  Note that when $r=s$,
$$A_p(q,B^r,r)=D_2(q,B^r).$$
The function $g$ will be fixed throughout the section, so we can drop the dependence on $g$. All implicit constants will be independent of $g$, $\delta$. We will spend the rest of this section proving the following key result. The main philosophy behind our proof is to gradually increase the size of the balls and then to decrease the size of the intervals.

\begin{theorem}
\label{djbcjhdfbvcjfvhfdghvfdgvh}
\label{nt27}Fix $n\ge 3$ and let $p<n(n+1)$ be sufficiently close to $n(n+1)$. Assume Theorem \ref{ntmain} holds for all smaller values of $n$, including $n=2$. Then for each $W>0$, there exist finitely many  positive weights $b_I,\gamma_I$ satisfying
$$\sum_I\gamma_I<1,$$
$$\sum_Ib_I\gamma_I>W,$$
so that for each sufficiently small $u>0$ (how small will only depend on $W$), the following inequality holds for each $0<\delta\le 1$ and for each ball $B^n$ in $\R^n$ of radius $\delta^{-n}$
\begin{equation}
\label{ne67}
A_p(u,B^n,u)\lesssim_{\epsilon,K,W}\delta^{-\epsilon}V_{p,n}(\delta)^{1-\sum_I\gamma_I}D_p(1,B^n)^{1-\sum_I\gamma_I}\times
\end{equation}
$$\prod_{I}A_p(ub_I,B^n,ub_I)^{\gamma_I}.$$
\end{theorem}

In the next subsection we prove this theorem for  $n=3$. Then we reinforce the ideas behind the proof in the case $n=4$, before explaining the case of arbitrary dimension.
\bigskip

\subsection{The case $n=3$}
\label{sub1}
In this subsection we will operate under the knowledge that Theorem \ref{ntmain} holds true for $n=2$, as proved in \cite{BD3}. In other words, we will use that
\begin{equation}
\label{ne50}
V_{6,2}(\delta)\lesssim_\epsilon \delta^{-\epsilon}.
\end{equation}
Alternatively, one may reprove this fact by a rather straightforward adaptation of the arguments in this paper. The proof is fairly short, and it involves only one type of ball inflation. The details are left to the interested reader.

We fix an integer $K\ge 2M_3$ and the intervals
$I_1,\ldots,I_{M_3}$ of the form $[\frac{i}{K},\frac{i+1}{K}]$, with $i$ in a collection of non consecutive integers among $0,1,\ldots,K-1$. Note that there are finitely many such choices of intervals, so the implicit constants in the inequalities in this section will not depend on the intervals we are using.

From now on, we will denote by $B^r$ an arbitrary ball of radius $\delta^{-r}$ in $\R^3$.

\begin{proposition}
\label{np4}
For $p\ge 6$, let $\alpha_1=\alpha_1(p)$ satisfy
$$\frac1{\frac{p}{3}}=\frac{\alpha_1}{\frac{2p}{3}}+\frac{1-\alpha_1}{2}.$$
Fix a  ball $B^2\subset \R^3$ with radius $\delta^{-2}$, and let $\B_1(B^2)$ be a finitely overlapping cover of $B^2$ with balls $B^1$. The following inequality holds
\begin{equation}
\label{rovgoprty8ut9v0vir890g76ry890u}
\big(\frac1{|\B_1(B^2)|}\sum_{B^1\in\B_1(B^2)}\big[\prod_{i=1}^{M_3}(\sum_{J_{i,1}\subset I_i\atop{|J_{i,1}|=\delta}}\|E_{J_{i,1}}g\|_{L^{2}_{\sharp}(w_{B^1})}^2)^{1/2}\big]^{\frac{p}{M_3}
}\big)^{\frac1p}
\lesssim_{\epsilon,K} \delta^{-\epsilon}
\end{equation}
$$\times \big[\prod_{i=1}^{M_3}(\sum_{J_{i,2}\subset I_i\atop{|J_{i,2}|=\delta^2}}\|E_{J_{i,2}}g\|_{L^{2}_{\sharp}(w_{B^2})}^2)^{1/2}\big]^{\frac{1-\alpha_1}{M_3}
}$$
$$\times \big[\prod_{i=1}^{M_3}(\sum_{J_{i,1}\subset I_i\atop{|J_{i,1}|=\delta}}\|E_{J_{i,1}}g\|_{L^{\frac{2p}3}_{\sharp}(w_{B^2})}^2)^{1/2}\big]^{\frac{\alpha_1}{M_3}
}.$$
\end{proposition}
\begin{proof}
Since $p\ge 6$, we have that $$\|E_{J_{i,1}}g\|_{L^{2}_{\sharp}(w_{B^1})}\le \|E_{J_{i,1}}g\|_{L^{\frac{p}{3}}_{\sharp}(w_{B^1})}.$$
Using this and  Theorem \ref{nt4} with $k=1$,  the term \eqref{rovgoprty8ut9v0vir890g76ry890u} is now controlled by
$$
\lesssim_{\epsilon,K}\delta^{-\epsilon} (\prod_{i=1}^{M_3}\sum_{J_{i,1}\subset I_i\atop{|J_{i,1}|=\delta}}\|E_{J_{i,1}}g\|_{L_{\sharp}^{\frac{p}3}(w_{B^2})}^2)^{\frac1{2M_3}}$$
Using  H\"older's inequality we can dominate this by
$$\le (\prod_{i=1}^{M_3}\sum_{J_{i,1}\subset I_i\atop{|J_{i,1}|=\delta}}\|E_{J_{i,1}}g\|_{L_{\sharp}^{2}(w_{B^2})}^2)^{\frac{1-\alpha_1}{2M_3}}(\prod_{i=1}^{M_3}\sum_{J_{i,1}\subset I_i\atop{|J_{i,1}|=\delta}}\|E_{J_{i,1}}g\|_{L_{\sharp}^{\frac{2p}3}(w_{B^2})}^2)^{\frac{\alpha_1}{2M_3}}.$$
It suffices now to apply $L^2$ decoupling via Lemma \ref{nl3} to the first term from above.

\end{proof}
\medskip

The sequence of inequalities from Proposition \ref{np4} can be summarized as follows
$$A_p(1,B^2,1)\lesssim_{\epsilon,K}\delta^{-\epsilon} D_{\frac{p}{3}}(1,B^2)\lesssim_{\epsilon,K}\delta^{-\epsilon}$$
\begin{equation}
\label{ne6}
\times A_p(2,B^2,2)^{1-\alpha_1}
\end{equation}
\begin{equation}
\label{ne6'}
\times D_{\frac{2p}{3}}(1,B^2)^{\alpha_1}.
\end{equation}
We have so far performed one ball inflation, replacing $B^1$ with $B^2$. A few preliminary remarks are in order. Note that the terms of type $A_p(u,B,v)$ have $u=v$. This symmetry will always be preserved throughout the iteration and will facilitate the transition from Proposition \ref{np6} to Theorem \ref{nt7}. The reason we build our iteration scheme around the pivotal terms  $A_p(v,B,v)$ may seem rather subtle at this point, but we hope it will become more transparent throughout the argument. Suffices to say that we could have alternatively worked with other pivotal terms (see for comparison our argument in \cite{BD4}), however the use of $A_p(v,B,v)$ will make the $L^2$ decoupling rather streamlined.

\medskip

We will illustrate \eqref{ne6}, \eqref{ne6'}  using the following tree.
\bigskip

\centerline
{\bf Figure 1}
$$
\begin{CD}
A_p (1, B^1, 1)\\
 @VVV @.\\
D_{\frac p3} (1, B^2)\\
^{1-\alpha_1}\!\!\swarrow \qquad\quad \searrow^{\!\!\alpha_1}\\[-4pt]
A_p(2, B^2, 2) \qquad\qquad D_{\frac {2p}3} (1, B^2)
\end{CD}
$$
\medskip

In the next stage we process the term \eqref{ne6'}. We proceed with a second ball inflation, shifting from  balls $B^2$ to balls $B^3$ and invoking the multilinear Kakeya-type inequality \eqref{ne3} with $k=2$. We continue with a lower dimensional decoupling followed by $L^2$ decoupling.  The term \eqref{ne6} will not undergo any serious modification, it will simply become $A_p(2,B^3,2)^{1-\alpha_1}$.
For $p\ge 9$ let $\alpha_2,\beta_2\in [0,1]$ satisfy
$$\frac1{\frac{2p}3}=\frac{1-\alpha_2}{6}+\frac{\alpha_2}{p},$$
$$\frac16=\frac{1-\beta_2}2+\frac{\beta_2}{\frac{2p}3}.$$

\begin{proposition}
\label{np3}
Fix a ball $B^3\subset \R^3$ with radius $\delta^{-3}$. Let $\B_2(B^3)$ be a finitely overlapping cover of $B^3$ with balls $B^2$. For each $B^2$, let $\B_1(B^2)$ be a finitely overlapping cover of $B^2$ with balls $B^1$. Define
$$\B_1(B^3)=\{B^1:\;B_1\in\B_1(B^2),\;B^2\in\B_2(B^3)\}.$$
Then
$$\big(\frac1{|\B_1(B^3)|}\sum_{B^1\in\B_1(B^3)}\big[\prod_{i=1}^{M_3}(\sum_{J_{i,1}\subset I_i\atop{|J_{i,1}|=\delta}}\|E_{J_{i,1}}g\|_{L^{2}_{\sharp}(w_{B^1})}^2)^{1/2}\big]^{\frac{p}{M_3}
}\big)^{\frac1p}\lesssim_{\epsilon,K} \delta^{-\epsilon}$$
$$\times\big(\frac1{|\B_2(B^3)|}\sum_{B^2\in\B_2(B^3)}\big[\prod_{i=1}^{M_3}(\sum_{J_{i,2}\subset I_i\atop{|J_{i,2}|=\delta^2}}\|E_{J_{i,2}}g\|_{L^{2}_{\sharp}(w_{B^2})}^2)^{1/2}\big]^{\frac{p}{M_3}
}\big)^{\frac{1-\alpha_1}p}$$
$$\times \big[\prod_{i=1}^{M_3}(\sum_{J_{i,3}\subset I_i\atop{|J_{i,3}|=\delta^3}}\|E_{J_{i,3}}g\|_{L^{2}_{\sharp}(w_{B^3})}^2)^{1/2}\big]^{\frac{\alpha_1(1-\alpha_2)(1-\beta_2)}{M_3}}$$

$$\times \big[\prod_{i=1}^{M_3}(\sum_{J_{i,\frac32}\subset I_i\atop{|J_{i,\frac32}|=\delta^{\frac32}}}\|E_{J_{i,{\frac32}}}g\|_{L^{\frac{2p}3}_{\sharp}(w_{B^3})}^2)^{1/2}\big]^{\frac{\alpha_1(1-\alpha_2)\beta_2}{M_3}
}$$
$$\times \big[\prod_{i=1}^{M_3}(\sum_{J_{i,1}\subset I_i\atop{|J_{i,1}|=\delta}}\|E_{J_{i,1}}g\|_{L^p_{\sharp}(w_{B^3})}^2)^{1/2}\big]^{\frac{\alpha_1\alpha_2}{M_3}
}.$$
\end{proposition}
\begin{proof}
There are four different stages in the argument.
\medskip

In the first one we average inequality \eqref{ne6}-\eqref{ne6'} raised to the power $p$ over all balls $B^2\in\B_2(B^3)$. Then use H\"older's inequality and Theorem \ref{nt4} with $k=2$ to get
$$A_p(1,B^3,1)^p\lesssim_{\epsilon,K} $$$$\delta^{-\epsilon}A_p(2,B^3,2)^{p(1-\alpha_1)}\big[ \frac1{|\B_2(B^3)|}\sum_{B^2\in\B_2(B^3)}D_{\frac{2p}{3}}(1,B^2)^p\big]^{\alpha_1}\lesssim_{\epsilon,K}$$$$\delta^{-\epsilon}
A_p(2,B^3,2)^{p(1-\alpha_1)} D_{\frac{2p}{3}}(1,B^3)^{p\alpha_1}.$$
We thus can write
\begin{equation}
\label{ne7}
A_p(1,B^3,1)\lesssim_{\epsilon,K}\delta^{-\epsilon}A_p(2,B^3,2)^{1-\alpha_1} D_{\frac{2p}{3}}(1,B^3)^{\alpha_1}.
\end{equation}
The second step is simply H\"older's inequality
\begin{equation}
\label{ne8}
D_{\frac{2p}{3}}(1,B^3)\le D_6(1,B^3)^{1-\alpha_2}D_{p}(1,B^3)^{\alpha_2}.
\end{equation}
To justify the use of $L^6$ we move to the third stage of the argument. The key observation is that the $\delta^3-$neighborhood of each arc $\Gamma_J$ on the curve $(t,t^2,t^3)$ with $|J|=\delta$ is essentially a $\delta^3-$neighborhood of an arc of similar length on a rigid motion of a parabola $(t,t^2)$. By applying lower dimensional decoupling, Lemma \ref{nl5} with $n=k=3$ and $p=6$,  we can write (recalling \eqref{ne50})
\begin{equation}
\label{ne9}
D_6(1,B^3)\lesssim_\epsilon \delta^{-\epsilon}D_6(\frac32,B^3).\end{equation}
In the fourth stage of the argument we use H\"older
\begin{equation}
\label{ne10}
D_6(\frac32,B^3)\le D_2(\frac32,B^3)^{1-\beta_2}D_{\frac{2p}{3}}(\frac32,B^3)^{\beta_2}
\end{equation}
and $L^2$ decoupling
\begin{equation}
\label{ne11}
D_2(\frac32,B^3)\lesssim D_{2}(3,B^3).
\end{equation}
It suffices now to combine \eqref{ne7}--\eqref{ne11}.

\end{proof}
We can summarize the new inequality as follows
$$A_p(1,B^3,1)\lesssim_{\epsilon,K}\delta^{-\epsilon}\times$$
\begin{equation}\label{ne12}
A_p(2,B^3,2)^{1-\alpha_1}\times\end{equation}\begin{equation}
\label{ne14}A_p(3,B^3,3)^{\alpha_1(1-\alpha_2)(1-\beta_2)}\times\end{equation}\begin{equation}
\label{ne15}D_{\frac{2p}{3}}(\frac32,B^3)^{\alpha_1(1-\alpha_2)\beta_2}\times\end{equation}\begin{equation}
\label{ne16}D_p(1,B^3)^{\alpha_1\alpha_2}.
\end{equation}
Let us briefly compare this with its previous incarnation, \eqref{ne6}-\eqref{ne6'}.
Note that the term \eqref{ne6} has become \eqref{ne12} without undergoing any processing. Most importantly, the term \eqref{ne6'} has reincarnated into the three terms \eqref{ne14}--\eqref{ne16}, via the use of ball inflation, $L^2$ decoupling and lower dimensional decoupling.
We will illustrate this new inequality  using the following tree.

\medskip
\centerline
{\begin{minipage}{4.9in} 
\centerline
{\includegraphics[height=3.2in]{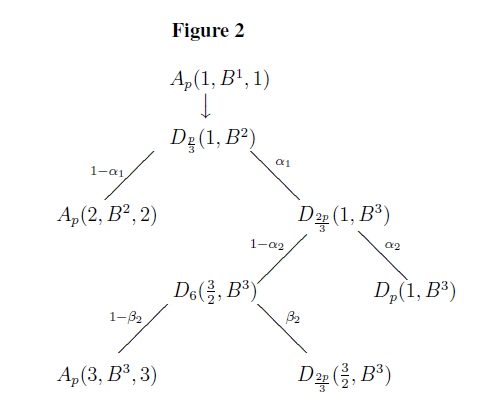} }
\end{minipage}}

\medskip

\bigskip

The four leaves in the tree correspond to the terms \eqref{ne12}--\eqref{ne16}, and the weights are multiplied along the edges to create the exponents in \eqref{ne12}-\eqref{ne16}.
\bigskip

This completes the basic step of our iteration scheme. So far we have performed two ball inflations. The original one was by a (logarithmic) factor of 2, from $B^1$ to $B^2$. We will not perform such an inflation again. The second one was by a factor of $\frac32$, from $B^2$ to $B^3$. This type of inflation will appear in each step of the iteration.

Let us explain how to iterate one more time. We increase again the size of the ball, by replacing $B^3$ with $B^{\frac32\cdot3}$. We sum \eqref{ne12}--\eqref{ne16} raised to the power $p$ over a finitely overlapping cover $\B_3(B^{\frac92})$ of $B^{\frac92}$ with balls $B^{3}$, and proceed exactly as in the proof of Proposition \ref{np3} to get
$$A_p(1,B^{\frac92},1)^p\lesssim_{\epsilon,K}\delta^{-\epsilon}\times$$
\begin{equation}\label{ne17}
A_p(2,B^{\frac92},2)^{p(1-\alpha_1)}\times\end{equation}\begin{equation}
\label{ne18} A_p(3,B^{\frac92},3)^{p\alpha_1(1-\alpha_2)(1-\beta_2)}\times\end{equation}\begin{equation}
\label{ne19}\big[\frac1{|\B_3(B^{\frac92})|}\sum_{B^3\in\B_3(B^{\frac92})}D_{\frac{2p}{3}}(\frac32,B^3)^p\big]^{\alpha_1(1-\alpha_2)\beta_2}\times\end{equation}\begin{equation}
\label{ne20}\big[\frac1{|\B_3(B^{\frac92})|}\sum_{B^3\in\B_3(B^{\frac92})}D_p(1,B^3)^p\big]^{\alpha_1\alpha_2}.
\end{equation}
By using \eqref{j nbyrtug8rt7g856uy956y9} for the term \eqref{ne20} and Theorem \ref{nt4} with $k=2$ for \eqref{ne19}, we can further  dominate the above by
\begin{equation}\label{ne21}
A_p(2,B^{\frac92},2)^{p(1-\alpha_1)}\times\end{equation}\begin{equation}
\label{ne22}A_p(3,B^{\frac92},3)^{p\alpha_1(1-\alpha_2)(1-\beta_2)}\times\end{equation}\begin{equation}
\label{ne23}D_{\frac{2p}{3}}(\frac32,B^{\frac92})^{p\alpha_1(1-\alpha_2)\beta_2}\times\end{equation}\begin{equation}
\label{ne24}D_p(1,B^{\frac92})^{p\alpha_1\alpha_2}.
\end{equation}
The only term that needs further processing is \eqref{ne23}, and this is done following the argument in stages 2-4 of the proof of Proposition \ref{np3}. More exactly, we can write
$$
D_{\frac{2p}{3}}(\frac32,B^{\frac92})\lesssim_{\epsilon}\delta^{-\epsilon}\times
$$
\begin{equation}
\label{ne25}
A_p(\frac92,B^{\frac92},\frac92)^{(1-\alpha_2)(1-\beta_2)}\times
\end{equation}\begin{equation}
\label{ne26}
D_{\frac{2p}{3}}(\frac94,B^{\frac92})^{(1-\alpha_2)\beta_2}\times
\end{equation}
\begin{equation}
\label{ne27}
D_p(\frac32,B^{\frac92})^{\alpha_2}.
\end{equation}
Putting together \eqref{ne17}--\eqref{ne27} we arrive at the following new version of our main inequality
$$A_p(1,B^{\frac92},1)\lesssim_{\epsilon,K}\delta^{-\epsilon}\times$$
\begin{equation*}
A_p(2,B^{\frac92},2)^{1-\alpha_1}\times\end{equation*}\begin{equation*}
A_p(3,B^{\frac92},3)^{\alpha_1(1-\alpha_2)(1-\beta_2)}\times\end{equation*}
\begin{equation*}
A_p(\frac92,B^{\frac92},\frac92)^{(1-\alpha_2)(1-\beta_2)\alpha_1(1-\alpha_2)\beta_2}\times
\end{equation*}
\begin{equation*}
D_{\frac{2p}{3}}(\frac94,B^{\frac92})^{(1-\alpha_2)\beta_2\alpha_1(1-\alpha_2)\beta_2}\times
\end{equation*}
\begin{equation*}
D_p(1,B^{\frac92})^{\alpha_1\alpha_2}\times
\end{equation*}
\begin{equation*}
D_p(\frac32,B^{\frac92})^{\alpha_2\alpha_1(1-\alpha_2)\beta_2}.
\end{equation*}

This completes the second iteration. The tree associated with this inequality is as follows. The six terms on the right hand side of the inequality from above correspond to the six leaves of the tree. There are additional features of this  tree that will be discussed in the end of this subsection.
\medskip

\centerline
{\begin{minipage}{4.9in} 
\centerline
{\includegraphics[height=4.2in]{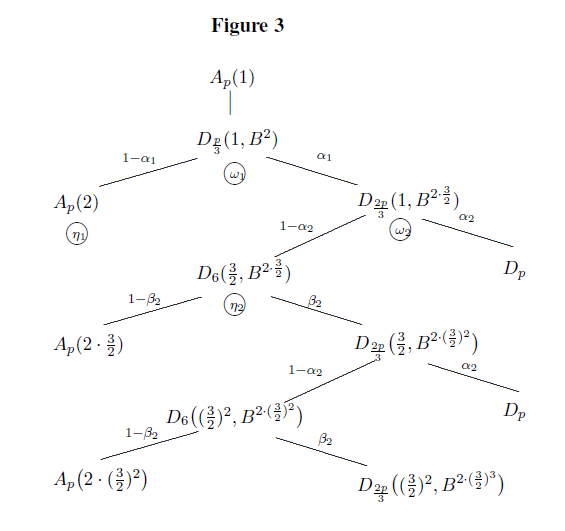} }
\end{minipage}}

\medskip

We can iterate this further, by increasing each time the exponent of the radius of the ball with a factor of $\frac32$, and processing the term $D_{\frac{2p}{3}}$ using ball inflation, lower dimensional decoupling and $L^2$ decoupling. After $r$ iterations we get the following inequality
$$
A_p(1,B,1)\lesssim_{\epsilon,K,r}\delta^{-\epsilon}A_p(2,B,2)^{1-\alpha_1}\prod_{i=1}^rA_p(2(\frac32)^i,B,2(\frac32)^i)^{\alpha_1(1-\alpha_2)(1-\beta_2)[(1-\alpha_2)\beta_2]^{i-1}}$$
$$\times D_{\frac{2p}3}((\frac32)^r,B)^{\alpha_1[(1-\alpha_2)\beta_2]^r}$$
$$\times\prod_{i=0}^{r-1}D_p((\frac32)^i,B)^{\alpha_1\alpha_2[(1-\alpha_2)\beta_2]^i},
$$
for each ball $B\subset \R^3$ with radius $\delta^{-2(\frac32)^r}$ and $ p\ge 9$.
\bigskip

Next, invoke H\"older to argue that
$$D_{\frac{2p}3}((\frac32)^r,B)\lesssim D_{p}((\frac32)^r,B),$$
and use this inequality to write
$$
A_p(1,B,1)\lesssim_{\epsilon,K,r}\delta^{-\epsilon}A_p(2,B,2)^{1-\alpha_1}\prod_{i=1}^rA_p(2(\frac32)^i,B,2(\frac32)^i)^{\alpha_1(1-\alpha_2)(1-\beta_2)[(1-\alpha_2)\beta_2]^{i-1}}$$
$$\times\prod_{i=0}^{r-1}D_p((\frac32)^i,B)^{\alpha_1\alpha_2[(1-\alpha_2)\beta_2]^i}D_{p}((\frac32)^r,B)^{\alpha_1[(1-\alpha_2)\beta_2]^r}
.$$

Note that by raising this to the power $p$ and summing over finitely overlapping families of balls, the above inequality holds in fact for all balls $B$ of radius at least $\delta^{-2(\frac32)^r}$. By renaming the variable $\delta$, we can rewrite the new inequality as follows.
\begin{proposition}
\label{np5}
Let $p>9$.
Let $u>0$ be such that
$$(\frac32)^ru\le 1.$$
Then for each ball $B^3$ of radius $\delta^{-3}$ we have
$$
A_p(u,B^3,u)\lesssim_{\epsilon,K,r}\delta^{-\epsilon}D_{p}((\frac32)^ru,B^3)^{\alpha_1[(1-\alpha_2)\beta_2]^r}\prod_{i=0}^{r-1}D_p((\frac32)^iu,B^3)^{\alpha_1\alpha_2[(1-\alpha_2)\beta_2]^i}\times
$$
$$A_p(2u,B^3,2u)^{1-\alpha_1}\prod_{i=1}^rA_p(2(\frac32)^iu,B^3,2(\frac32)^iu)^{\alpha_1(1-\alpha_2)(1-\beta_2)[(1-\alpha_2)\beta_2]^{i-1}}.$$
\end{proposition}
To simplify notation, we will write
$$\gamma_0=1-\alpha_1$$
$$\gamma_i=\alpha_1(1-\alpha_2)(1-\beta_2)[(1-\alpha_2)\beta_2]^{i-1},\;1\le i\le r$$
$$b_i=2(\frac32)^{i},\;0\le i\le r.$$
We will also use the trivial fact that
\begin{equation}
\label{ne42}
D_p((\frac32)^iu,B^3)\le V_p(\delta)D_p(1,B^3),\;0\le i\le r.
\end{equation}
Note that this inequality involves no rescaling.
Some elementary computations show that the inequality in Proposition \ref{np5} can now be rewritten as follows.
\begin{proposition}
\label{np6}
Let $p>9$.
Let $u>0$ be such that
$$(\frac32)^ru\le 1.$$
Then for each ball $B^3$ of radius $\delta^{-3}$ we have
$$
A_p(u,B^3,u)\lesssim_{\epsilon,K,r}\delta^{-\epsilon}V_p(\delta)^{1-\sum_{0}^r\gamma_j}D_p(1,B^3)^{1-\sum_{0}^r\gamma_j}\times
$$
$$\prod_{j=0}^rA_p(b_ju,B^3,b_ju)^{\gamma_j}.$$
\end{proposition}
Note that $1-\sum_{j=0}^r\gamma_j>0$, since all terms of type $D_p$ from Proposition \ref{np5} have a positive exponent.
\bigskip

The next step is to iterate Proposition \ref{np6}.
We start by noting that if the following stronger condition holds
$$2(\frac32)^{2r}u\le 1,$$
then $u'=b_iu$ satisfies the condition
$$(\frac32)^ru'\le 1,$$
for each $0\le i\le r$. Thus by applying Proposition \ref{np6} with $u'$ replacing $u$ we can also write
$$
A_p(b_iu,B^3,b_iu)\lesssim_{\epsilon,K,r}\delta^{-\epsilon}V_p(\delta)^{1-\sum_{0}^r\gamma_j}D_p(1,B^3)^{1-\sum_{0}^r\gamma_j}\times
$$
$$\prod_{j=0}^rA_p(b_jb_iu,B^3,b_jb_iu)^{\gamma_j}.$$
Combining these estimates with the inequality in Proposition \ref{np6} we get
$$
A_p(u,B^3,u)\lesssim_{\epsilon,K,r}\delta^{-\epsilon}V_p(\delta)^{1-(\sum_{0}^r\gamma_j)^2}D_p(1,B^3)^{1-(\sum_{0}^r\gamma_j)^2}\times
$$
$$\prod_{j=0}^r\prod_{i=0}^rA_p(b_jb_iu,B^3,b_jb_iu)^{\gamma_j\gamma_i}.$$
By iterating this process $M$ times we arrive at the following result.
\begin{theorem}
\label{nt7}
Let $p>9$. Given the integers $r,M\ge 1$,
let $u>0$ be such that
$$[2(\frac32)^r]^Mu\le 2
.$$
Then for each ball $B^3$ of radius $\delta^{-3}$ and $p\ge 9$ we have
$$
A_p(u,B^3,u)\lesssim_{\epsilon,K,r,M}\delta^{-\epsilon}V_p(\delta)^{1-(\sum_{0}^r\gamma_j)^M}D_p(1,B^3)^{1-(\sum_{0}^r\gamma_j)^M}\times
$$
$$\prod_{j_1=0}^r\ldots\prod_{j_M=0}^rA_p(u\prod_{l=1}^Mb_{j_l},B^3,u\prod_{l=1}^Mb_{j_l})^{\prod_{l=1}^M\gamma_{j_l}}.$$
\end{theorem}
\bigskip
An easy computation shows that
$$\lim_{r\to\infty}\sum_{j=0}^r b_j\gamma_j=\frac{9}{p-3}(1+\frac{12-p}{p^2-12p+18}).$$
It is easy to see that
\begin{equation}
\label{ne41}
\sum_{j=0}^\infty b_j\gamma_j>1,
\end{equation}
for $p$ sufficiently close to, but less than the critical exponent $12.$. This observation concludes the proof of Theorem \ref{nt27} for $n=3$, by taking $b_I=\prod_{l=1}^Mb_{j_l}$ , $\gamma_I=\prod_{l=1}^M\gamma_{j_l}$ and $M$ large enough.

While \eqref{ne41} is  easily seen to be true via direct computations in the case $n=3$, a more abstract argument will be needed to address the similar question in higher dimensions.
In preparation for that, we close the subsection with a different perspective on the quantity $\sum_{j=0}^\infty b_j\gamma_j$. Combining this perspective with the result in the Appendix will immediately verify \eqref{ne41} for all $n\ge 3$.

To present this new perspective, let us take another look at the tree from Figure 3. We should in fact envision this as being part of an infinite  tree $\textbf{T}$, corresponding to an infinite number of iterations.
Note the weights $\omega_1,\omega_2,\eta_1, \eta_2$ on the tree $\textbf{T}$. The tree in Figure 3 focuses attention only on the information that was  relevant for  our derivation of Proposition \ref{np6}. First, the leaves $A_p(2(\frac32)^j)$ on the left side of $\textbf{T}$ are abbreviations of the terms $A_p(2(\frac32)^j,B,2(\frac32)^j)$.
This way we  encode the fact that such a term contributes with weight $b_j=2(\frac32)^j$ to the sum \eqref{ne41}.

Second, the leaves $D_p$ on the right side of $\textbf{T}$ do not undergo any processing throughout the iteration, other than being associated with increasingly larger balls. The way they are estimated at the end of the argument is simply using the definition of $V_p$, with no  rescaling (the gain via  rescaling is actually negligible, and only complicates the argument). See \eqref{ne42}.

Third, note the way we specify the entries of the terms $D_{\frac{2p}{3}}$. The size of the ball $2(\frac32)^j$ specifies the precise scale at which that particular term is being processed with Theorem \ref{nt4}. Note that each time such a term $D_{\frac{2p}{3}}((\frac32)^{j-1},B^{2(\frac32)^j})$ is processed this way, it gets replaced with three other terms:

$\bullet$ the similar term $D_{\frac{2p}{3}}((\frac32)^{j},B^{2(\frac32)^{j+1}})$, note the increment $j\mapsto j+1$

$\bullet$ $A_p(2(\frac32)^j)$, via $L^2$ decoupling,

$\bullet$ a term of type $D_p$.
\bigskip

The quantity $\omega_1=\omega_1(p)$ attached to the root of $\textbf{T}$ is
$$\omega_1:=\sum_{j=0}^\infty b_j\gamma_j.$$
The root collects contributions to the sum \eqref{ne41} from both of its  bifurcations. Of course, only the terms $A_p(2(\frac32)^j)$, $j\ge 0$, will produce contributions.
For example, the contribution from the first left leaf is $2(1-\alpha_1)$, which corresponds to the term $b_0\gamma_0$ in \eqref{ne41}. We write $\eta_1=2$, to denote the contribution coming from  $A_p(2)$.

Similarly, we denote by $\omega_2$ the contribution coming from the part of $\textbf{T}$ that is rooted at $D_{\frac{2p}{3}}(1,B^{2\cdot\frac32})$. We can write our first equation as follows
\begin{equation}\label{ne44}
\omega_1=(1-\alpha_1)\eta_1+\alpha_1\omega_2.
\end{equation}
Similarly, let $\eta_2$ be the contribution coming from the part of $\textbf{T}$ that is rooted at $D_{6}(\frac32,B^{2\cdot\frac32})$. Noting that the $D_p$ term on the right of $\textbf{T}$ produces no contribution, we can write the second equation as follows
\begin{equation}\label{ne45}
\omega_2=(1-\alpha_2)\eta_2+\alpha_2\omega_3,
\end{equation}
with the understanding that $\omega_3=0$.

Let us now write an equation for $\eta_2$. Note that $\eta_2$ collects contributions from the left and the right branches. The contribution from  the left comes with weight $1-\beta_2$ and it equals $\frac32\eta_1$, due to the self similarity of $\textbf{T}$ and the ball inflation. The one from the right is $\beta_2\frac32\omega_2$, due to same reasons.  Thus, the third and final equation is
\begin{equation}\label{ne46}
\eta_2=(1-\beta_2)\frac32\eta_1+\beta_2\frac32\omega_2.
\end{equation}
We summarize the equations \eqref{ne44}-\eqref{ne46} into the following system, when $n=3$
\begin{equation}
\label{ne47}
\begin{cases}\omega_j=(1-\alpha_j)\eta_j+\alpha_j\omega_{j+1}&:\quad 1\le j\le n-1,\;\omega_n=0 \\ \hfill \eta_j=(1-\beta_j)\frac{j+1}{j}\eta_{j-1}+\beta_j\frac{j+1}{j}\omega_{j}&:\quad 2\le j\le n-1,\;\eta_1=2 \end{cases}.
\end{equation}

This system will also describe the higher dimensional case $n\ge 4$, as will become clear in the following subsections. And again, it can be easily solved by hand for small values on $n$. In particular, when $n=3$ one gets
$$\omega_1(p)=\frac{9}{p-3}(1+\frac{12-p}{p^2-12p+18}),$$
as observed before. For higher values of $n$, the analysis will be done in the Appendix.
\bigskip

\subsection{The case $n=4$}

It will help the reader if we briefly describe the iteration scheme for $n=4$, before we move on to the general case. The new feature is that, in addition to the initial ball inflation by a factor of 2,  there will be two ball inflations that we will perform repeatedly: by a factor of $\frac32$ and by a factor of $\frac43$. We will follow the notation and philosophy from Subsection \ref{sub1}, with the obvious necessary modifications. For example, we replace $M_3=3!$ with $M_4=4!$ in the definitions of the terms of type $A_p$ and $D_p$.

The relevant Lebesgue indices in the case $n=4$ are $\frac{p}{4}, \frac{2p}{4}, \frac{3p}{4}, \frac{4p}{4}=p$ and also $2=2\cdot1$, $6=3\cdot 2$, $12=4\cdot 3$ and $20=5\cdot 4$. The relevant weights are $\alpha_1,\alpha_2,\alpha_3$ and $\beta_2,\beta_3$, specified by the equalities
$$\frac1{\frac{p}4}=\frac{1-\alpha_1}{2}+\frac{\alpha_1}{2\frac{p}{4}},$$
$$\frac1{\frac{2p}4}=\frac{1-\alpha_2}{6}+\frac{\alpha_2}{3\frac{p}{4}},$$
$$\frac1{\frac{3p}4}=\frac{1-\alpha_3}{12}+\frac{\alpha_3}{p},$$
$$\frac1{6}=\frac{1-\beta_2}{2}+\frac{\beta_2}{2\frac{p}4},$$
$$\frac1{12}=\frac{1-\beta_3}{6}+\frac{\beta_3}{3\frac{p}4}.$$

Throughout this subsection we will implicitly assume that $p$ is sufficiently close to the critical index $20=5\cdot 4$, so that all weights are in $[0,1]$.

Most importantly, we will operate under the assumption that Theorem \ref{ntmain} holds true for $n=3$, $n=2$. In other words, we will assume that
\begin{equation}
\label{ne48}
V_{6,2}(\delta)\lesssim_\epsilon \delta^{-\epsilon},
\end{equation}
\begin{equation}
\label{ne49}
V_{12,3}(\delta)\lesssim_\epsilon \delta^{-\epsilon}.
\end{equation}
Figure 4 below is of the same type as Figure 3, except that it only presents one iteration, containing one of each of the   three types of ball inflation:
$$B^1\mapsto B^2\mapsto B^3\mapsto B^4.$$
The way to continue this tree is to bifurcate each of the terms of type $D_{\frac{p}{2}}$ and $D_{\frac{3p}{p}}$, according to the rules specified in Figure 4. As we will make clear below, the next term that will be processed will be the leaf $D_{\frac{p}2}(\frac32,B^{3\cdot\frac32})$. That will happen when we increment from $B^4$ to $B^{3\cdot\frac32}$.

\medskip
\centerline
{\begin{minipage}{6.2in} 
\centerline
{\includegraphics[height=4.9in]{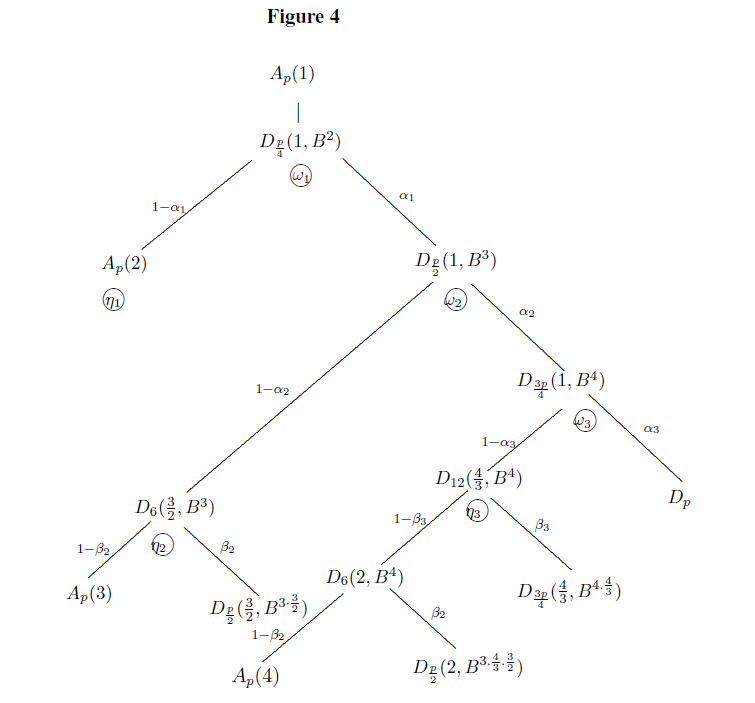} }
\end{minipage}}

\bigskip

We start with the following analog of Proposition \ref{np4}.
\begin{proposition}
\label{np10}
For each ball $B^2\subset \R^4$ with radius $\delta^{-2}$ we have
\begin{equation}
\label{ne63}
A_p(1,B^2,1)\lesssim_{\epsilon,K}\delta^{-\epsilon}D_{\frac{p}{4}}(1,B^2)\lesssim_{\epsilon,K}\delta^{-\epsilon}
A_p(2,B^2,2)^{1-\alpha_1}
D_{\frac{2p}{4}}(1,B^2)^{\alpha_1}.
\end{equation}
\end{proposition}
We next perform the first ball inflation, in order  to process the term $D_{\frac{2p}{4}}(1,B^2)$, by averaging \eqref{ne63} over a finitely overlapping cover $\B_2(B^3)$ of $B^3$ with balls $B^2$.
\begin{proposition}
\label{np11}
For each ball $B^3\subset \R^4$ with radius $\delta^{-3}$ we have
$$A_p(1,B^3,1)\lesssim_{\epsilon,K}\delta^{-\epsilon}\times$$
\begin{equation}\label{ne112}
A_p(2,B^3,2)^{1-\alpha_1}\times\end{equation}\begin{equation}
\label{ne114}A_p(3,B^3,3)^{\alpha_1(1-\alpha_2)(1-\beta_2)}\times\end{equation}\begin{equation}
\label{ne115}D_{\frac{2p}{4}}(\frac32,B^3)^{\alpha_1(1-\alpha_2)\beta_2}\times\end{equation}\begin{equation}
\label{ne116}D_{\frac{3p}4}(1,B^3)^{\alpha_1\alpha_2}.
\end{equation}
\end{proposition}
\begin{proof}
The argument follows closely the one in the proof of Proposition \ref{np3}. We first use Theorem \ref{nt4} with $k=2$ to replace $D_{\frac{2p}{4}}(1,B^2)$ with $D_{\frac{2p}{4}}(1,B^3)$. Then we interpolate (H\"older) $\frac{2p}{4}$ between $6$ and $\frac{3p}{4}$. Finally, we use \eqref{ne48} and lower dimensional decoupling for the $L^6$ term, and $L^2$ decoupling to arrive at  the $A_p$ term.

\end{proof}
There is only one thing left to do to finish the first iteration, that is to process the term \eqref{ne116}. To do so, we increase from $B^3$ to $B^4$. Let $\B_3(B^4)$ be a finitely overlapping cover of $B^4$ with balls $B^3$.
\begin{proposition}
\label{np12}
For each ball $B^4\subset \R^4$ with radius $\delta^{-4}$ we have
$$A_p(1,B^4,1)\lesssim_{\epsilon,K}\delta^{-\epsilon}\times$$
\begin{equation}\label{ne212}
A_p(2,B^4,2)^{1-\alpha_1}\times\end{equation}\begin{equation}
\label{ne214}A_p(3,B^4,3)^{\alpha_1(1-\alpha_2)(1-\beta_2)}\times\end{equation}\begin{equation}
\label{ne216}A_p(4,B^4,4)^{\alpha_1\alpha_2(1-\alpha_3)(1-\beta_3)(1-\beta_2)}\times
\end{equation}\begin{equation}
\label{ne215}\big[\frac{1}{|\B_3(B^4)|}\sum_{B^3\in\B_3(B^4)}D_{\frac{2p}{4}}(\frac32,B^3)^p\big]^{\frac1{p}\alpha_1(1-\alpha_2)\beta_2}\times\end{equation}
\begin{equation}
\label{ne217}D_{\frac{2p}{4}}(2,B^4)^{\alpha_1\alpha_2(1-\alpha_3)(1-\beta_3)\beta_2}\times
\end{equation}
\begin{equation}
\label{ne218}D_{\frac{3p}{4}}(\frac43,B^4)^{\alpha_1\alpha_2(1-\alpha_3)\beta_3}\times
\end{equation}
\begin{equation}
\label{ne219}D_p(1,B^4)^{\alpha_1\alpha_2\alpha_3}.
\end{equation}
\end{proposition}
\begin{proof}
Average the inequality from  Proposition \ref{np11} raised to the power $p$, over all the balls in  $\B_3(B^4)$.
Use Theorem \ref{nt4} with $k=3$, then use H\"older to interpolate $\frac{3p}{4}$ between $12$ and $p$, lower dimensional decoupling in $\R^3$ for the $D_{12}$ term (use \eqref{ne49}), H\"older again to interpolate $12$ between 6 and $\frac{3p}4$, lower dimensional decoupling  in $\R^2$ for the $D_6$ term (use \eqref{ne48}) and finally, $L^2$ decoupling for the $A_p$ term.

\end{proof}
This ends the first iteration. The seven terms  \eqref{ne212}-\eqref{ne219} correspond to the seven leaves of the tree in Figure 4.

Note that \eqref{ne115} does not fully get processed in this first iteration, and becomes \eqref{ne215}. However, this term will be processed on balls $B^{3\cdot \frac32}$, at the next ball inflation, where the increment is from $B^4$ to $B^{3\cdot \frac32}$. The argument goes as follows. Let $\B_4(B^{3\cdot \frac32})$ be a finitely overlapping cover of $B^{3\cdot \frac32}$ with balls $B^4$. Let $\B_3(B^4)$ be a finitely overlapping cover of $B^4$ with balls $B^3$. Then
$$\left(\frac{1}{|\B_4(B^{3\cdot \frac32})|}\sum_{B^4\in\B_4(B^{3\cdot \frac32})}\left(\big[\frac{1}{|\B_3(B^4)|}\sum_{B^3\in\B_3(B^4)}D_{\frac{2p}{4}}(\frac32,B^3)^p\big]^{\frac1{p}}\right)^p\right)^{1/p}=$$
\begin{equation}
\label{ nvbyo8uer897g98t8579-5y0}
\big[\frac{1}{|\B|}\sum_{B^3\in\B}D_{\frac{2p}{4}}(\frac32,B^3)^p\big]^{\frac1{p}},
\end{equation}
where $$\B=\bigcup_{B^4\in \B_4(B^{3\cdot \frac32})}\{B^3:\;B^3\in\B_3(B^4)\}$$
is a finitely overlapping cover of $B^{3\cdot\frac32}$ with balls $B^3$. Using Theorem \ref{nt4} with $n=4$ and $k=2$, the term \eqref{ nvbyo8uer897g98t8579-5y0} is dominated by
$$\lesssim_{\epsilon,K}\delta^{-\epsilon}D_{\frac{2p}{4}}(\frac32,B^{3\cdot\frac32}).$$
\bigskip

Imagine now an infinite version $\textbf{T}$ of the tree in Figure 4. The only leaves that we would see are the $A_p$ terms on the left and the $D_p$ terms on the right. The terms $D_p$ get processed in each stage of the iteration via \eqref{j nbyrtug8rt7g856uy956y9}. Also, all the terms of type $D_{\frac{2p}{4}}$ or $D_{\frac{3p}{4}}$ get eventually processed.  If we count all the  balls that take part in the inflation, we will find all radii of the form $\delta^{-3(\frac32)^i(\frac43)^j}$ with $i,j\ge 0$, in addition to the initial radii $\delta^{-1}$ and $\delta^{-2}$. The inflations occur in increasing order of these radii. We have explained the following sequence of inflations
$$1\mapsto 2\mapsto 2\cdot\frac32\mapsto 3\cdot\frac43\mapsto 3\cdot\frac32.$$
A term of type $D_{\frac{2p}{4}}$ gets processed on a ball $B^{3(\frac32)^i(\frac43)^j}$, if its path back to the root of $\textbf{T}$ encounters exactly $i$ terms of type $D_{\frac{2p}{4}}$ and exactly $j$ terms of type $D_{\frac{3p}{4}}$. A term of type $D_{\frac{3p}{4}}$ gets processed on a ball $B^{3(\frac32)^i(\frac43)^j}$, if its path back to the root of $\textbf{T}$ encounters exactly $i+1$ terms of type $D_{\frac{2p}{4}}$ and exactly $j-1$ terms of type $D_{\frac{3p}{4}}$.

Each time a term of type $D_{\frac{2p}{4}}$ is processed, it gets replaced with three terms of  type $A_p,D_{\frac{2p}{4}}, D_{\frac{3p}{4}}$. Each time a term of type $D_{\frac{3p}{4}}$ is processed, it gets replaced with four terms of  type $A_p,D_{\frac{2p}{4}}, D_{\frac{3p}{4}}, D_{\frac{4p}{4}}$.
\bigskip

While this process may be complicated, the features that are relevant for us turn out to be rather simple. We summarize them in the following analog of Proposition \ref{np6}.

\begin{proposition}
\label{np16}
Assume \eqref{ne48} and \eqref{ne49}, and let $p$ be sufficiently close to $20$. Let $u>0$ be sufficiently close to $0$, depending on how large $r$ is (the exact dependence is irrelevant).
Then for each ball $B^4$ of radius $\delta^{-4}$ in $\R^4$ we have
$$
A_p(u,B^4,u)\lesssim_{\epsilon,K,r}\delta^{-\epsilon}V_{p,4}(\delta)^{1-\sum_{0}^r\gamma_j}D_p(1,B^4)^{1-\sum_{0}^r\gamma_j}\times
$$
$$\prod_{j=0}^rA_p(b_ju,B^4,b_ju)^{\gamma_j},$$
for some $\gamma_j,b_j>0$.  Also, we have the crucial identity
\begin{equation}
\label{ne60}
\omega_1:=\sum_{j=0}^\infty b_j\gamma_j,
\end{equation}
where $\omega_1=\omega_1(p)$ is the solution of the system \eqref{ne47} with $n=4$.
\end{proposition}
\begin{proof}
Apart from \eqref{ne60}, the statement is rather tautological. Indeed, using simple inequalities such as \eqref{j nbyrtug8rt7g856uy956y9},
$$D_{\frac{2p}4}(u,B^4),D_{\frac{3p}4}(u,B^4)\lesssim D_{p}(u,B^4)$$
and
$$D_p(v,B^4)\le V_{p,4}(\delta)D_p(1,B^4),\;v\ge 1$$
we obviously end up only with terms of type $A_p(b_ju,B^4,b_ju)$, $D_p(1,B^4)$, $V_p(\delta)$, for appropriate $b_j$, whose values will not be important. The resulting inequality will necessarily be of the type
$$
A_p(u,B^4,u)\lesssim_{\epsilon,K,r}\delta^{-\epsilon}V_{p,4}(\delta)^{1-D_r}D_p(1,B^4)^{1-C_r}\times
$$
$$\prod_{j=0}^rA_p(b_ju,B^4,b_ju)^{\gamma_j},$$
for some $C_r\le D_r$. Since $V_{p,4}(\delta)\ge 1$, we might as well replace $D_r$ with $C_r$.
Whatever the values of $\gamma_j$ are, the fact that $C_r=\sum_{j=0}^r\gamma_j$ follows from scaling considerations, by recalling the definition of the $A_p,D_p$ terms.
\medskip

The identity \eqref{ne60} follows from the observations we made earlier, exactly like in the case $n=3$, which was explained at the end of the previous subsection.

\end{proof}

By iterating Proposition \ref{np16} $M$ times, in the same way we did it for  $n=3$, we arrive at the following inequality.
\begin{theorem}Let $p$ be sufficiently close to $20$.
\label{nt17}Assume \eqref{ne48} and \eqref{ne49} hold. Let $r, M\ge 1$. Then for $u>0$ sufficiently close to $0$ (how small depends on $r$ and $M$) and for each ball $B^4$ in $\R^4$ of radius $\delta^{-4}$  we have
$$
A_p(u,B^4,u)\lesssim_{\epsilon,K,r,M}\delta^{-\epsilon}V_{p,4}(\delta)^{1-(\sum_{0}^r\gamma_j)^M}D_p(1,B^4)^{1-(\sum_{0}^r\gamma_j)^M}\times
$$
$$\prod_{j_1=0}^r\ldots\prod_{j_M=0}^rA_p(u\prod_{l=1}^Mb_{j_l},B^4,u\prod_{l=1}^Mb_{j_l})^{\prod_{l=1}^M\gamma_{j_l}}.$$
Moreover, $\gamma_j,b_j$ satisfy \eqref{ne60}.
\end{theorem}
Combining this with Theorem \ref{tApe} from the Appendix finishes the proof of Theorem \ref{nt27} for $n=4$, by taking $b_I=\prod_{l=1}^Mb_{j_l}$ , $\gamma_I=\prod_{l=1}^M\gamma_{j_l}$ and $M$ large enough.

\bigskip

\subsection{The case of arbitrary $n$}
\label{slast}
A careful reading of the previous two subsections allows for a rather straightforward extension of our  iteration scheme to all dimensions.
We will have $n-1$ types of increments in the ball inflation, by factors $\frac{j+1}j$, $1\le j\le n-1$. The one corresponding to $j=1$ is performed only once, at the very beginning of the argument. The ones corresponding to $j\ge 2$ will be performed repeatedly.
\bigskip

The relevant Lebesgue indices are now $\frac{jp}{n},\;2\le j\le n$ and also $j(j+1),\;1\le j\le n$. The relevant weights are $\alpha_1,\ldots,\alpha_{n-1}$ and $\beta_2,\ldots,\beta_{n-1}$. Their exact values are
$$\alpha_j=\frac{\frac{p}{n}-(j+1)}{\frac{p}{n}-j},$$
$$\beta_j=\frac{\frac{2p}{n}}{(j+1)(\frac{p}{n}-j+1)}.$$
Throughout this section we will implicitly assume that $p$ is sufficiently close to the critical index $n(n+1)$, so that all weights are in $[0,1]$.

Most importantly, we will operate under the assumption that Theorem \ref{ntmain} holds true  in all dimensions smaller than $n$. In other words, we will assume that
\begin{equation}
\label{ne65}
V_{j(j+1),j}(\delta)\lesssim_\epsilon \delta^{-\epsilon},\;\;2\le j\le n-1,\;0<\delta\le 1.
\end{equation}

Figure 5 below is of the same type as Figure 4 from previous subsection.
\medskip

\bigskip

\centerline
{\begin{minipage}{6.9in} 
\centerline
{\includegraphics[height=4.9in]{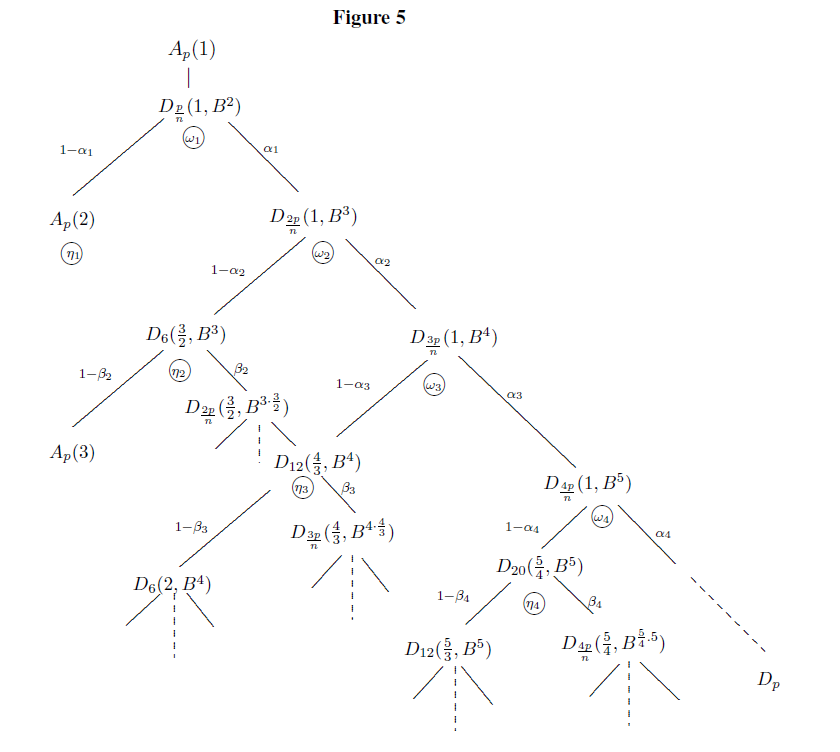} }
\end{minipage}}

\medskip

\medskip

\begin{theorem}
\label{nt27aaa}Let $p$ be sufficiently close to $n(n+1)$. Assume \eqref{ne65}. Let $r,M\ge 1$. Then for $u>0$ sufficiently close to $0$, for each $0<\delta\le 1$ and for each ball $B^n$ in $\R^n$ of radius $\delta^{-n}$  we have
$$
A_p(u,B^n,u)\lesssim_{\epsilon,K,r,M}\delta^{-\epsilon}V_{p,n}(\delta)^{1-(\sum_{0}^r\gamma_j)^M}D_p(1,B^n)^{1-(\sum_{0}^r\gamma_j)^M}\times
$$
$$\prod_{j_1=0}^r\ldots\prod_{j_M=0}^rA_p(u\prod_{l=1}^Mb_{j_l},B^n,u\prod_{l=1}^Mb_{j_l})^{\prod_{l=1}^M\gamma_{j_l}}.$$
Moreover, we have the crucial identity
$$
\omega_1:=\sum_{j=0}^\infty b_j\gamma_j,
$$
where $\omega_1=\omega_1(p)$ is the solution of the system \eqref{ne47}.
\end{theorem}

The proof of Theorem \ref{nt27aaa} is virtually identical to the one of its lower dimensional counterparts presented in the earlier subsections. Combining this with Theorem \ref{tApe} from the Appendix finishes the proof of Theorem \ref{nt27} for arbitrary $n$, by taking $b_I=\prod_{l=1}^Mb_{j_l}$ , $\gamma_I=\prod_{l=1}^M\gamma_{j_l}$ and $M$ large enough.

\bigskip

\section{The proof of Theorem \ref{ntmain} for $n\ge 3$}
\label{nsproof}

In this section we prove Theorem \ref{ntmain} for $n\ge 3$. For $p\ge 2$, let $\eta_p\ge 0$ be the unique number such that
\begin{equation}\label{ne33}
\lim_{\delta\to 0}V_{p,n}(\delta)\delta^{\eta_p+\sigma}=0,\;\text{for each }\sigma>0
\end{equation}
and\begin{equation}\label{ne34}
\limsup_{\delta\to 0}V_{p,n}(\delta)\delta^{\eta_p-\sigma}=\infty,\;\text{for each }\sigma>0.
\end{equation}
It is immediate that $\eta_p$ is a finite number.
\bigskip

We start by presenting the following rather immediate consequence of Theorem \ref{nt27}. It will be helpful to start indexing the terms $A_p,D_p$ by the function $g$.
\begin{theorem}
\label{nt27agga}Fix $n\ge 3$ and let $p<n(n+1)$ be sufficiently close to $n(n+1)$. Assume Theorem \ref{ntmain} holds for all smaller values of $n$. Then for each $W>0$, for each sufficiently small $u>0$, the following inequality holds for each $g:[0,1]\to\C$,  $0<\delta\le 1$ and for each ball $B^n$ in $\R^n$ of radius $\delta^{-n}$
$$
A_p(u,B^n,u,g)\lesssim_{\epsilon,K, W}\delta^{-(1-uW)(\eta_p+\epsilon)}D_p(1,B^n,g).
$$
\end{theorem}
\begin{proof}Let $b_I,\gamma_I$ be the weights corresponding to $W$ from Theorem \ref{nt27}. Recall that we have for all small enough $u>0$
\begin{equation}
\label{ne67bygtytt78ty87t}
A_p(u,B^n,u,g)\lesssim_{\epsilon,K,W}\delta^{-\epsilon}V_{p,n}(\delta)^{1-\sum_I\gamma_I}D_p(1,B^n,g)^{1-\sum_I\gamma_I}\prod_{I}A_p(ub_I,B^n,ub_I,g)^{\gamma_I}.
\end{equation}

Assume in addition that $u$ is so small that $uW<1$ and $ub_I<1$ for each $I$.
First, H\"older's inequality shows that
\begin{equation}\label{nnne36}A_p(b_Iu,B^n,b_Iu,g)\lesssim D_p(b_Iu,B^n,g).
\end{equation}
Next, rescaling as in \eqref{ne31} leads to
\begin{equation}\label{ne36}
D_p(b_Iu,B^n,g)\lesssim V_{p,n}(\delta^{1-ub_I} )D_p(1,B^n,g).
\end{equation}
It suffices now to combine \eqref{ne33}, \eqref{ne67bygtytt78ty87t}, \eqref{nnne36} and \eqref{ne36}.

\end{proof}

Recall that we need to prove that
\begin{equation}
\label{ne70}
V_{n(n+1),n}(\delta)\lesssim_{\epsilon,p} \delta^{-\epsilon}.
\end{equation}
It turns out that it suffices to prove a similar statement with $n(n+1)$ replaced by $p<n(n+1)$.

\begin{lemma}Inequality \eqref{ne70} follows if
\begin{equation}
\label{ne7000000000000000}
V_{p,n}(\delta)\lesssim_{\epsilon,p} \delta^{-\epsilon}
\end{equation}
holds for each $p<p_n=n(n+1)$ sufficiently close to $n(n+1)$.
\end{lemma}
\begin{proof}
Let $B\subset \R^n$ be a ball with radius $\delta^{-n}$.
Using  inequality \eqref{ nghbugtrt90g0-er9t-9}, for  $p<p_n$ we have
$$\|E_{[0,1]}g\|_{L^{p_n}(w_{B})}\lesssim \|E_{[0,1]}g\|_{L^{p}(w_{B})}.$$
Combining this with H\"older's inequality we get
$$\|E_{[0,1]}g\|_{L^{p_n}(w_{B})}\lesssim V_{p,n}(\delta)(\sum_{J\subset [0,1]\atop{|J|=\delta}}\|E_Jg\|_{L^{p}(w_B)}^2)^{1/2}$$
$$\lesssim V_{p,n}(\delta)\|\textbf{1}\|_{L^{\frac{q}{q-1}}(w_B)}(\sum_{J\subset [0,1]\atop{|J|=\delta}}\|E_Jg\|_{L^{p_n}(w_B)}^2)^{1/2}.$$
It suffices to note that $q\to 1$ as $p\to p_n$.

\end{proof}

We now return to proving \eqref{ne70}. The proof is by induction. Recall that the case $n=2$ was proved in \cite{BD3}. Assume that we have proved \eqref{ne70} for all values less than $n$. Using the previous lemma, it will suffice to prove \eqref{ne7000000000000000}.

For the rest of the argument fix $p<p_n$ so that Theorem \ref{nt27agga} holds.
Recall that we need to prove that $\eta_p=0$.
\bigskip

Fix $K\ge 2M_n$, $0<\delta\le K^{-1}$, and $I_1,\ldots,I_{M_n}$, as in the previous section. Let  $W, u$ be as in Theorem \ref{nt27agga}.  Let $B^n$ be an arbitrary ball with radius $\delta^{-n}$ in $\R^n$.

Start by applying Cauchy--Schwarz
$$
\|(\prod_{i=1}^{M_n}E_{I_i}g)^{1/M_n}\|_{L^{p}_{\sharp}(w_{B^n})}\le \delta^{-\frac{u}2}\|(\prod_{i=1}^{M_n}\sum_{J_{i,u}\subset I_i\atop{|J_{i,u}|=\delta^{u}}}|E_{J_{i,u}}g|^2)^{\frac1{2M_n}}\|_{L^{p}_{\sharp}(w_{B^n})}$$
\begin{equation}
\label{ne35}
\lesssim \delta^{-\frac{u}2}\big(\frac1{|\B_u(B^n)|}\sum_{B^u\in\B_u(B^n)}\|(\prod_{i=1}^{M_n}\sum_{J_{i,u}\subset I_i\atop{|J_{i,u}|=\delta^{u}}}|E_{J_{i,u}}g|^2)^{\frac1{2M_n}}\|_{L^{p}_{\sharp}(w_{B^u})}^p\big )^{\frac1p}.
\end{equation}
Here $\B_u(B^n)$ is a finitely overlapping cover of $B^n$ with balls $B^u$ of radius $\delta^{-u}$.
Using Minkowski's inequality, \eqref{ne35} is dominated by $$\delta^{-\frac{u}2}\big(\frac1{|\B_u(B^n)|}\sum_{B^u\in\B_u(B^n)}D_p(u,B^u,g)^p\big)^{1/p}.$$ On the other hand, inequality \eqref{ nghbugtrt90g0-er9t-9} shows that $D_p(u,B^u,g)\lesssim D_2(u,B^u,g)$, and consequently
$$\big(\frac1{|\B_u(B^n)|}\sum_{B^u\in\B_u(B^n)}D_p(u,B^u,g)^p\big)^{1/p}\lesssim A_p(u,B^n,u,g).$$
We conclude that
\begin{equation}
\label{irurggype pge pgphgt  uyuuy}
\|(\prod_{i=1}^{M_n}E_{I_i}g)^{1/M_n}\|_{L^{p}_{\sharp}(w_{B^n})}\le \delta^{-\frac{u}2}A_p(u,B^n,u,g).\end{equation}
Next, use Theorem \ref{nt27agga}, \eqref{ne33} and \eqref{irurggype pge pgphgt  uyuuy} to write for each $\sigma>0$
$$\|(\prod_{i=1}^{M_n}E_{I_i}g)^{1/M_n}\|_{L^{p}_{\sharp}(w_{B^n})}\lesssim_{\sigma,\epsilon,K,W}\delta^{-(\epsilon+\frac{u}{2}+(\eta_p+\sigma)(1-uW))}D_p(1,B^n,g).$$
Note that this inequality holds uniformly over  all $g$ and $B^n$, so we can take the supremum over these elements to get
$$V_p(\delta,K)\lesssim_{\sigma,\epsilon,K,W}\delta^{-(\epsilon+\frac{u}{2}+(\eta_p+\sigma)(1-uW))}.$$
 Using \eqref{ne29} and \eqref{ne34} we can now write
$$
\delta^{-\eta_p+\epsilon_p(K)+\sigma}\lesssim_{\sigma,\epsilon,K,W}\delta^{-(\epsilon+\frac{u}{2}+(\eta_p+\sigma)(1-uW))}.
$$
Since this holds true for $\delta$ arbitrarily close to $0$, it further leads to
\begin{equation*}\label{ne37}
\eta_p-\epsilon_p(K)-\sigma\le \epsilon+\frac{u}{2}+(\eta_p+\sigma)(1-uW),\end{equation*}
which can be rewritten as follows
$$\eta_p\le \frac{1}{2W}+\frac{\epsilon+\epsilon_p(K)+\sigma(2-uW)}{uW}.$$
This holds true for each $\epsilon,\sigma>0$ and each $K\ge 2M_n$. Recalling that
$$\lim_{K\to\infty}\epsilon_p(K)=0,$$ we can further write
$$\eta_p\le \frac{1}{2W}.$$
Letting now $W\to\infty$ leads to $\eta_p=0$, as desired.
\bigskip

\section{Appendix}
Fix $n\ge 3$. For $\Delta$ in a small neighborhood of $n+1$, let
$$\alpha_j=\frac{\Delta-(j+1)}{\Delta-j},\;\;1\le j\le n-1,$$
$$\beta_j=\frac{2\Delta}{(j+1)(\Delta-j+1)},\;\;2\le j\le n-1.$$
We will be concerned with proving the following result, which was instrumental in achieving the conclusion of Theorem \ref{nt27}.

\begin{theorem}
\label{tApe}
For $\Delta<n+1$ sufficiently close to $n+1$, the solution $\omega_1(\Delta)$ of
the system
\begin{equation}
\label{ne147}
\begin{cases}\omega_j=(1-\alpha_j)\eta_j+\alpha_j\omega_{j+1}&:\quad 1\le j\le n-1\\ \hfill \eta_j=(1-\beta_j)\frac{j+1}{j}\eta_{j-1}+\beta_j\frac{j+1}{j}\omega_{j}&:\quad 2\le j\le n-1 \\  \omega_n=0\\\eta_1=2\end{cases}
\end{equation}
satisfies
$$\omega_1(\Delta)>1.$$
\end{theorem}
\begin{proof}
For arbitrary $\theta\in\R$, consider the related system
\begin{equation}
\label{ne247}
\begin{cases}\omega_j=(1-\alpha_j)\eta_j+\alpha_j\omega_{j+1}&:\quad 1\le j\le n-1\\ \hfill \eta_j=(1-\beta_j)\frac{j+1}{j}\eta_{j-1}+\beta_j\frac{j+1}{j}\omega_{j}&:\quad 2\le j\le n-1 \\  \omega_n=\theta\\\eta_1=2,\end{cases}
\end{equation}
and call $\omega_j(\Delta,\theta),\eta_j(\Delta,\theta)$ its solution.

Iterating the first equation in \eqref{ne247} a few times leads to
\begin{equation}
\label{ne80}
\omega_j=\frac1{\Delta-j}(\eta_j+\eta_{j+1}+\ldots+\eta_{n-1})+\frac{\Delta-n}{\Delta-j}\theta,\;1\le j\le n-1.
\end{equation}
Substituting \eqref{ne80} in the second equation from \eqref{ne247} leads to
\begin{equation}
\label{ne81}
\eta_j= (1-\frac{2\Delta}{(j+1)(\Delta-j+1)})\frac{j+1}{j}\eta_{j-1}+
\end{equation}
$$\frac{2\Delta}{j(\Delta-j)(\Delta-j+1)}(\eta_j+\eta_{j+1}+\ldots+\eta_{n-1})+\frac{2\Delta(\Delta-n)}{j(\Delta-j+1)}\theta,$$ for $2\le j\le n-1.$

Together with the initial condition $\eta_1=2$, \eqref{ne81} determines all $\eta_j$. Then \eqref{ne80} determines all $\omega_j$.

It is an easy verification that
\begin{equation}
\label{ne82}
\eta_j(n+1,0)=2\omega_j(n+1,0)=2\frac{n-j}{n-1}
\end{equation}
and in fact more generally, for each $\Delta$ close to $n+1$
\begin{equation}
\label{ne83}
\eta_j(\Delta,\frac{\Delta-n-1}{\Delta-2})=2\omega_j(\Delta,\frac{\Delta-n-1}{\Delta-2})=2\frac{\Delta-j-1}{\Delta-2}.
\end{equation}
In particular,
\begin{equation}
\label{ne84}
\omega_1(\Delta,\frac{\Delta-n-1}{\Delta-2})=1.
\end{equation}
It will suffice to prove that for each $\Delta$ sufficiently close to $n+1$, the function $\theta\mapsto\omega_1(\Delta,\theta)$ is increasing in a neighborhood of 0. Indeed, combining this with \eqref{ne84} for $\Delta<n+1$ sufficiently close to $n+1$  and with the fact that $\frac{\Delta-n-1}{\Delta-2}<0$ for such $\Delta$, we get
$$\omega_1(\Delta,0)>\omega_1(\Delta,\frac{\Delta-n-1}{\Delta-2})=1,$$
as desired.

Now,  \eqref{ne84} will follow if we prove that
\begin{equation}
\label{ne85}
\omega_1(\Delta,\theta)=A(\Delta)+B(\Delta)\theta
\end{equation}
with $B(\Delta)>0$, for $\Delta$ close to $n+1$. The linear representation is clear from  \eqref{ne80}, \eqref{ne81}. Invoking continuity of the rational function $B(\Delta)$, it will  suffice to argue that $B(n+1)>0$.

To simplify notation, we fix $\theta$ and denote by $\omega_j,\eta_j$ the solution $\omega_j(n+1,\theta),\eta_j(n+1,\theta)$ of the system \eqref{ne247} corresponding to $\Delta=n+1.$
The first two equations of the system \eqref{ne247}
can be written as follows
$$\omega_j=f_j(\eta_j,\omega_{j+1}),\;\;1\le j\le n-2$$
$$\omega_{n-1}=f_{n-1}(\eta_{n-1},\theta)$$
$$\eta_j=g_j(\eta_{j-1},\omega_j),\;\;2\le j\le n-2,$$
with  $f_j(\cdot,\cdot),g_j(\cdot,\cdot)$ linear. Moreover, the coefficients of all $f_j,g_j$ are strictly positive.
We can iterate these equations as many times as we wish, by each time paying attention to the new representation for $\omega_1$ as a function of $\omega_2,\ldots,\omega_{n-1},\theta,\eta_2,\ldots,\eta_{n-1}.$ For example, iterating twice amounts to writing
$$\omega_1=f_1(2,f_2(\eta_2,\omega_3)),$$
while iterating again leads to
$$\omega_1=f_1(2,f_2(g_2(2,\omega_2),f_3(\eta_3,\omega_{4}))).$$
Note that all these iterations lead to affine combinations. Write the result after $r$ iterations as follows
\begin{equation}
\label{ne86}
\omega_1=A_r+B_r\theta+a_{1,r}\omega_1+\ldots+a_{n-1,r}\omega_{n-1}+b_{2,r}\eta_2+\ldots b_{n-1,r}\eta_{n-1},
\end{equation}
with all coefficients $A_r,B_r,a_{i,r},b_{i,r}\ge 0$ and independent of $\theta$. It is fairly immediate to note that $B_r>0$ for $r\ge n-1$, and that $A_r,B_r$ are nondecreasing functions of $r$.

By using $\theta=0$ in \eqref{ne86}, combined with \eqref{ne82} and \eqref{ne84}, we find that
\begin{equation}
\label{ne88}
1=\omega_1(n+1,0)=A_r+\sum_{i=1}^{n-1}a_{i,r}\frac{n-i}{n-1}+2\sum_{i=2}^{n-1}b_{i,r}\frac{n-i}{n-1}\ge A_r.
\end{equation}
The key observation is that
\begin{equation}
\label{ne87}
\lim_{r\to\infty}(\sum_{i=1}^{n-1}a_{i,r}+\sum_{i=2}^{n-1}b_{i,r})=0.
\end{equation}
We will argue that if this were not true, it would force $A_r$ to approach $\infty$, contradicting \eqref{ne88}. Indeed, note that since $\eta_1>0$ and the coefficients of $f_j,g_j$ are $>0$, iterating the term $$a_{1,r}\omega_1+\ldots+a_{n-1,r}\omega_{n-1}+b_{2,r}\eta_2+\ldots b_{n-1,r}\eta_{n-1}$$
sufficiently many times will add to the value of $A_r$ a nonzero fraction of $$(\sum_{i=1}^{n-1}a_{i,r}+\sum_{i=2}^{n-1}b_{i,r}).$$
Finally, use \eqref{ne87} and let $r\to\infty$ in \eqref{ne86} to write
$$\omega_1=A+B\theta,$$
with $A:=\lim_{r\to\infty}A_r$, $B:=\lim_{r\to\infty}B_r$. The fact that $B>0$ follows from previous observations.

\end{proof}

\end{document}